\documentclass[a4paper]{article}

\usepackage{algorithmic}
\usepackage{algorithm}
\usepackage{amsmath}
\usepackage{amsfonts}
\usepackage{amssymb}
\usepackage{amsthm}
\usepackage{tikz}
\usepackage{pgf}
\usepackage{pgfplots,  pgfplotstable}
\pgfplotsset{compat=1.5}
\usepgfplotslibrary{external}
\usepgfplotslibrary{statistics}
\usetikzlibrary{decorations.pathreplacing}
\usepackage[utf8]{inputenc}
\usetikzlibrary{shapes,decorations,patterns}
\usepackage{siunitx} 

\usepackage{tabularx}
\usepackage{graphicx}
\usepackage{multirow}
\usepackage{booktabs}
\usepackage{url}

\newtheorem{theorem}{Theorem}

\newtheorem{lemma}{Lemma}

\newtheorem{definition}{Definition}
\newtheorem{corollary}{Corollary}
\usepackage{comment}
\usepackage[hidelinks]{hyperref} 
\usepackage[a4paper,includeheadfoot,margin=2.54cm]{geometry}
\title{ODDO: Online Duality-Driven Optimization}
\author{Martijn H. H. Schoot Uiterkamp, Marco E. T. Gerards, Johann L. Hurink \\ University of Twente, Enschede, the Netherlands}

\usepackage{subcaption}

\newtheorem*{lemma_tight}{Lemma~\ref{lemma_tight}}
\newtheorem*{lemma_prop_2}{Lemma~\ref{lemma_prop_2}}

\newcommand{\prob}{P}

\allowdisplaybreaks

\begin{document}
\maketitle
%\begin{multicols}{2}

\begin{abstract}
Motivated by energy management for micro-grids, we study convex optimization problems with uncertainty in the objective function and sequential decision making. To solve these problems, we propose a new framework called ``Online Duality-Driven Optimization'' (ODDO). This framework distinguishes itself from existing paradigms for optimization under uncertainty in its efficiency, simplicity, and ability to solve problems without any quantitative assumptions on the uncertain data. The key idea in this framework is that we predict, instead of the actual uncertain data, the optimal Lagrange multipliers. Subsequently, we use these predictions to construct an online primal solution by exploiting strong duality of the problem. We show that the framework is robust against prediction errors in the optimal Lagrange multipliers both theoretically and in practice. In fact, evaluations of the framework on problems with both real and randomly generated input data show that ODDO can achieve near-optimal online solutions, even when we use only elementary statistics to predict the optimal Lagrange multipliers.
\end{abstract}

\section{Introduction}

\subsection{Optimization under uncertainty in micro-grids}
Dealing with uncertainty plays an important role in modern decision making. As a consequence, many paradigms have been developed for optimization under data uncertainty (we discuss several of them later in this section). Each of these paradigms has its own pros and cons, and it depends strongly on the given application which of them is most suitable. In this article, we consider as leading example an application of optimization under uncertainty in so-called \emph{micro-grids}. Micro-grids are branches of the main distribution grid that aim to be as self-sustainable as possible and to this end aim to minimize the import of energy from the main grid. These micro-grids play an important role in the ongoing world-wide energy transition, where two trends can be observed in residential energy production and consumption that threaten the proper operation and security of the local energy grid. First, more and more of the required energy for the local grid is also generated locally from renewable sources, e.g., solar power. Second, there is an increase in energy demand due to an increasing electrification of devices such as electric vehicles (EVs) and heat pumps. Together, this means that both the demand and supply of energy within the local grid increases, which leads to more stress on these grids. In fact, field tests show that even a small number of EVs can cause serious damage to the grid when their charging is not being coordinated \cite{Hoogsteen2017b}.

One way to overcome this problem would be to simply increase the capacity of the grid by, e.g., installing more or thicker cables. However, this is an extremely costly operation (see, e.g., \cite{Haucap2013}), since this increase of capacity is mainly needed for accommodating sporadic peaks in demand and supply, which leads to a lower utilization rate. A different and less costly way is to treat the local energy system as a micro-grid where the energy flows are actively managed, thereby reducing, among others, the peaks in demand and supply. Within the micro-grid, one can exploit the flexibility of devices such as EVs to compensate for consumption and production peaks by adjusting their consumption if desired. This type of energy management is called \emph{demand side management} (DSM). To exploit the device flexibility, many DSM approaches compute coordinated energy consumption schedules for all steerable devices such that the resulting aggregated energy profile (e.g., on the neighborhood level) is as flat as possible \cite{Esther2016}.

Computing these schedules requires information on energy production and consumption in the upcoming time period, i.e., data that is unknown at the current time. In many DSM approaches, this data is simply predicted and a deterministic version of the scheduling problem at hand is considered based on this predicted data \cite{Barbato2014}. However, in practice, predicting especially the base load consumption of all non-steerable devices is hard since it is to some extend equivalent to the prediction of human behavior on a small scale (see, e.g., \cite{Javed2012}). For this reason, current research directions in this area consider the development of advanced prediction approaches \cite{Hong2016} and the application of existing frameworks for optimization under data uncertainty (see, e.g., \cite{Ma2018}).

%In the remainder of this section, we provide an overview of such existing paradigms for optimization under uncertainty (Section~\ref{sec_uncertain}), formulate the general optimization problem considered in this article (Section~\ref{sec_prob}), and summarize our contributions (Section~\ref{sec_contribute}).

\subsection{Existing paradigms for optimization under uncertainty}
\label{sec_uncertain}

The optimization problems considered in this article are concerned with sequential (or, alternatively, online or multistage) decision making under uncertainty. There exists a vast literature on such optimization problems (see \cite{Bakker2020} for a survey). Here, we briefly discuss three of the most prominent paradigms for solving these problems and their advantages and disadvantages, namely (multistage) stochastic programming, (adjustable) robust optimization, and online convex optimization.

Stochastic programming (SP) is used to solve optimization problems where the uncertain data is assumed to be stochastic under a given joint probability distribution. The goal is to design a policy that maximizes the expected outcome of some function on the decision variables and uncertain data under the constraint that all or almost all constraints are fulfilled for all considered data outcomes \cite{Shapiro2009}. In particular, in a \emph{multi-stage} SP, the aim is to iteratively derive a policy for a single (the current) stage and a set of policies for the future stages for each outcome of the stochastic data of the current stage.

The advantage of SP is that all information on the probability distribution is taken into account in the decision making, which leads to more accurate results. However, a major drawback of SP is that knowledge of the distribution underlying the data is not always available and cannot always be estimated well from historical data, e.g., due to the data not being stationary. Moreover, multi-stage SPs are generally intractable \cite{Dyer2006} and, although several approaches exist to compute approximate solutions (see, e.g., \cite{Shapiro2005}), the exponential growth of the dimension of the to-be-solved SP problem is a limiting factor for applications where fast computations are crucial.

A different approach is taken in robust optimization (RO), which, in contrast to SP, is a solution paradigm that does not assume any probability distribution on the uncertain data. Instead, the essential assumption in RO is that one can construct a so-called \emph{uncertainty set} that contains all possible or likely realizations of the uncertain data \cite{Ben-Tal2009}. Given such an uncertainty set, the goal is to compute the best worst-case solution to the problem, i.e., the best solution that is feasible for any realization of the data within the uncertainty set. A variant of RO for sequential or online decision making is \emph{adjustable} RO \cite{Yankoglu2019}. Here, one distinguishes between here-and-now variables, which have to be decided at the current stage, and wait-and-see variables, which can or must be decided in later stages.

The major advantage of RO and adjustable RO is that any quantitative information on the uncertain data can be taken into account when constructing the uncertainty set. Moreover, the decision maker can control the level of conservatism involved in the process of determining a solution by adjusting the uncertainty set. However, a disadvantage of RO is that its solutions can be too conservative since not all realizations within the uncertainty set are equally likely to occur. Indeed, it is known that a careful design of the uncertainty set is crucial for the succes of RO for solving real-life problems \cite{Gorissen2015}. Another disadvantage is that solving a RO-model is in general much harder than solving the original problem with deterministic data and greatly increases the size of the problem at hand \cite{Ben-Tal2015}.

A third considered paradigm is online convex optimization (OCO). In the OCO-framework \cite{Shalev-Shwartz2012}, the goal is to solve convex optimization problems with objective uncertainty by learning the behavior of the objective function ``on the go". At each stage, the then available information on the objective function is used to determine the decision for this stage. Subsequently, the part of the objective function corresponding to this stage is revealed and the resulting cost is incurred. 

The major advantage of this framework is that it requires no assumptions on the uncertainty in the objective function other than that this function is convex. Moreover, it provides performance guarantees in terms of both objective value and constraint violations \cite{Mahdavi2012}. However, some of these guarantees require that the optimization horizon is extremely large \cite{Jenatton2016}, thus rendering them not informative in practice. Moreover, the practical performance of OCO heavily depends on estimates of several properties of the uncertain objective function such as its Lipschitz constant. Such an estimate depends on the actual uncertain data underlying the objective function and thus requires knowledge on this data itself. In any case, the classical OCO-framework does not consider the availability of further information (e.g., obtained from measurements) on the part of the objective function corresponding to the current stage. It is only very recently that OCO-models are being studied with so-called \emph{1-step look-ahead} features, where such measurements are (partly) available (see, e.g., \cite{Ho-Nguyen2019}).

\subsection{Our contribution}
\label{sec_contribute}

In this article, we focus on a class of constrained convex optimization problems with uncertainty in the objective function. In these problems, the optimization horizon is divided into stages  and we associate with each stage a vector of decision variables and an uncertain cost function. At the start of a given stage, this uncertain cost function is revealed and the decision variables corresponding to this stage have to be decided. Depending on the application, this revelation can be based on, e.g., a measurement or (very) short-term forecast of the cost function. The goal is to obtain in this way a feasible solution to the problem whose objective value is close to the optimum of the deterministic version of the optimization problem. We call this decision making problem the \emph{online} version or setting of the optimization problem. Note that, when deciding about the variables for a given stage, the uncertain cost functions for the future stages should be taken into consideration since the constraints may involve variables of different stages.

The scheduling problems in energy management mentioned in the previous section fall into this problem class (see, e.g., \cite{vdKlauw2017}) and we come back to these problems in more detail in Section~\ref{sec_eval}. Moreover, there is quite some research that focuses especially on problems with uncertainty only in the objective function. Examples of such problems can be found in the fields of linear programming \cite{DeAngelis1979}, network flow problems \cite{Bertsimas2003}, distributed optimization \cite{Li2018}, and multi-objective optimization \cite{Ide2016}. Examples of applications besides energy management are task scheduling on processors \cite{Gerards2016aa}, production and inventory management \cite{Mula2006}, and speed optimization for truck and ship routing \cite{Hvattum2013}.

We propose a new framework for solving optimization problems with uncertain objective functions witin this class that overcomes the aforementioned disadvantages of SP, RO, and OCO. We call this framework ``Online Duality-Driven Optimization'' (ODDO). The key idea in ODDO is that we can characterize the optimal decision vector for a given stage using the optimal Lagrange multipliers of the deterministic optimization problem and the parameters corresponding to the given stage. In particular, for this characterization, we require no information on cost functions other than that of the given stage. As a consequence, there is no need to predict these future cost functions and/or the uncertain data underlying these functions. Instead, we only require predictions of the optimal Lagrange multipliers of the deterministic version of the problem. This is beneficial in problems where the dimension of the dual space, i.e., the number of constraints, is low compared to the dimension of the uncertain data since it reduces the number of values that need to be predicted. Furthermore, we argue that this approach is robust against prediction errors in the Lagrange multipliers. In particular, for the special class of separable optimization problems with submodular constraints (see, e.g., \cite{Hochbaum1995,Fujishige2005}), we derive several bounds on the difference in objective value between our online solution and the offline optimal solution.

The idea of predicting optimal Lagrange multipliers shares similarities with active set learning for continuous and mixed-integer optimization \cite{Misra2019,Bertsimas2020}. These works solve optimization problems by predicting the set of constraints that is active at the optimal solution. Subsequently, they use this set to reduce the original problem to a convex equality-constrainted optimization problem that, depending on the linearity of the constraints, might be easier to solve. Active set learning and ODDO are similar in the sense that both approaches predict sets of values (active constraints or Lagrange multipliers) from which an optimal solution can be easily reconstructed. In fact, since the values of optimal Lagrange multipliers directly imply the set of active constraints by complementary slackness (see, e.g., \cite{Boyd2004}), active set learning is able to reconstruct an optimal solution with less information than ODDO. However, solving the aforementioned reduced equality-constrained problem requires all uncertain cost functions to be known a priori. As a consequence, the approaches in \cite{Misra2019,Bertsimas2020} are only able to solve offline optimization problems and cannot be used to solve sequential optimization problems where uncertain data is revealed over time. In our approach, however, we reduce the original problem to a collection of smaller problems, one for each stage, that each can be solved without using input data that has not yet been revealed. This means that given a prediction of the optimal Lagrange multipliers, we can compute a solution for a given stage using this prediction and only data that has already been revealed.

It is worth mentioning that, in energy management, several specific optimization problems under uncertainty are already being solved by similar approaches, be it under a different name. For example, in the case of EV charging, this approach is known as \emph{online valley filling} \cite{Chen2014,Mou2015,Gerards2016a} and the single Lagrange multiplier that is predicted is called the \emph{fill-level}. Our approach generalizes these approaches to a setting with convex objective function and constraints.

Our approach has several advantages compared to the solution paradigms mentioned in Section~\ref{sec_uncertain}:
\begin{enumerate}
\item
We do not pose any quantitative assumptions on the uncertain data in terms of support, underlying probability distribution, or uncertainty set. Instead, we only assume that the data possesses some structure, meaning that the data is similar for, e.g., consecutive problem instances.
\item
A disadvantage of, e.g., SP and RO is the large increase in complexity of the problems that eventually have to be solved, e.g., the robust counterpart in RO. In our approach, however, at each stage, one has to solve at most the original problem with a slightly different, but still convex, objective function. The construction of this new objective function consists of adding a linear term to the objective. Thus, the online problem can be solved using the same optimization methodology that is used to solve the deterministic problem.
\item
Our approach does not only work for the case where underlying data is uncertain, but also when the structure of the cost functions (e.g., quadratic, exponential) is uncertain. The only information that we require of these functions is that they are convex and continuously differentiable on a given compact support set.
\item
The performance guarantee that we derive is independent of the number of stages and only depends on the prediction errors in the optimal Lagrange multipliers. As a consequence, as opposed to some works on OCO, this guarantee holds for any number of stages and makes our approach successfully applicable also to problems with only a small number of stages.
\item
The derivation and validation of our approach requires only well-known optimality and stability results from convex programming. This means that the approach is relatively easy to apply by practitioners.
\end{enumerate}

For the envisioned application domain of our framework, namely energy management, these advantages are relevant in the following way:
\begin{enumerate}
\item
We circumvent the problem of having to predict energy consumption data of individual households, which is in practice very difficult \cite{Javed2012} and already an entire discipline on its own \cite{Hong2016}.
\item
The limited increase in complexity of the online problem compared to the deterministic problem is relevant because the computation of device schedules within a micro-grid often has to be done on embedded systems with relatively low computational power \cite{Beaudin2015}.
\item
Dpending on new information received by the measured energy consumption data within the houses of the micro-grid, one may decide to change the operational mode, e.g., to reduce the peak consumption of a particular group of houses.
\item
The optimization horizon in micro-grid management typically consists of only a few days that are divided into time slots of 10 or 15 minutes (see, e.g., \cite{Markle-Hus2018}). This means that the total number of stages is only a few hundreds, whereas some of the performance guarantees and bounds in OCO, e.g., in \cite{Mahdavi2012} require that the number of stages is in the order of tens of millions. 
\item
The relative simplicity of our approach is relevant since energy management is a multi-disciplinary research area with many experts from engineering that might not have the time and/or mathematical knowledge to delve into the rather technical literature of, e.g., SP, RO, and OCO. 
\end{enumerate}

We evaluate the performance of the ODDO-framework by applying it to two types of problems. The first problem is a scheduling problem from the area of energy management, namely the scheduling of a large neighborhood battery such that the net peak consumption of the neighborhood is minimized. The second problem is an inventory management problem studied in \cite{Ben-Tal2004} that was used to introduce and demonstrate the concept of adjustable RO. Simulation results show that in both cases, ODDO is able to achieve near-optimal online solutions and that these solutions can be achieved using easy-to-compute predictions for the optimal Lagrange multipliers. Moreover, the results indicate that ODDO is able to significantly outperform optimization over nominal or expected values, i.e., sample average approximation.

The remainder of this article is organized as follows. In Section~\ref{sec_frame}, we formulate the studied class of optimization problems and present the ODDO-framework for solving these problems in an online setting. In Section~\ref{sec_rob}, we analyze the theoretical performance of the ODDO-framework and prove several results regarding its robustness to prediction errors in its optimal Lagrange multipliers. In Section~\ref{sec_eval}, we evaluate the performance of the ODDO-framework and in Section~\ref{sec_disc}, we discuss several limitations and possible extensions of the framework and suggestions for future research. Finally, Section~\ref{sec_concl} contains some concluding remarks.

\section{The ODDO-framework}
\label{sec_frame}
In this section, we present the ODDO-framework for solving Problem~\prob\ in an online setting. We start this section by formulating the studied optimization problem in Section~\ref{sec_problem} and formalizing the online decision making process. Subsequently, we revisit in Section~\ref{sec_Lagrange} several basic results from Lagrangian duality theory. In Section~\ref{sec_approach}, we present our approach. Note, that we postpone the motivation for and robustness analysis of the ODDO-framework to Section~\ref{sec_rob}. Finally, in Section~\ref{sec_toy}, we provide an illustrational example of the ODDO-framework.

\subsection{Problem formulation}
\label{sec_problem}

We consider a finite horizon consisting of $T$ stages indexed by the set $\mathcal{T} := \lbrace 1,\ldots,T \rbrace$. With each stage $t \in \mathcal{T}$, we associate an index set~$\mathcal{N}^t$ of size $N^t$ and a decision vector $x^t := (x^t_i)_{i \in \mathcal{N}^t}$ of dimension $N^t$. Furthermore, for each stage $t \in \mathcal{T}$, a continuously differentiable and strictly convex cost function $f^t (x^t)$ is given. The objective is to minimize the sum of the cost functions for each stage. We consider two types of constraints. First, we impose for each stage~$t$ that $x^t$ has to be chosen from a compact convex set $\mathcal{C}^t \subset \mathbb{R}^{N^t}$. Second, we consider both equality and inequality constraints that are separable over the stages, i.e., that are of the form
\begin{align*}
\sum_{t \in \mathcal{T}} g^t_j (x^t) &\leq 0, \quad j \in \mathcal{M}, \\
\sum_{t \in \mathcal{T}} h^t_k (x^t) &= 0, \quad k \in \mathcal{L}.
\end{align*}
Here, we assume that each function $g^t_j$ is convex and continuously differentiable and each function $h^t_k$ is affine. Summarizing, we study the following convex optimization problem:
\begin{subequations}
\begin{align}
\text{\prob} \ : \ \min_{x^1,\ldots,x^T} \ & \sum_{t \in \mathcal{T}} f^t (x^t) \label{P_01}\\
\text{s.t. } & \sum_{t \in \mathcal{T}} g^t_j (x^t) \leq 0, \quad j \in \mathcal{M}, \label{P_02} \\
& \sum_{t \in \mathcal{T}} h^t_k (x^t) = 0, \quad k \in \mathcal{L}, \label{P_03}\\
& x^t \in \mathcal{C}^t, \quad t \in \mathcal{T} . \label{P_04}
\end{align} \label{P}%
\end{subequations}
In particular, we study Problem~\prob\ in the online setting, i.e., the setting wherein the cost functions $f^t$ are uncertain. In this online version, the nature of the uncertainty can be related to either uncertain input parameters, e.g., uncertain coefficients of a given polynomial, or an uncertain function type, e.g., polynomial or exponential. In any case, we assume that the uncertainty is such that $f^t$ remains convex and continuously differentiable for any possible outcome of the uncertain data underlying the function or any possible structure of the function. Most importantly, however, \emph{we do not assume any quantitative knowledge of the uncertainty}. This includes knowledge in terms of, e.g., types of cost functions, probability distributions, and uncertainty sets. 

The decision making process in the online setting is as follows. At each stage $t \in \mathcal{T}$, the corresponding cost function $f^t$ is revealed and the decision maker must decide on the corresponding decision vector $x^t$. The goal is to obtain a good online solution $\hat{x} := (\hat{x}^1,\ldots,\hat{x}^T)$. In this context, ``good'' means that the objective value of the solution is close to the optimum of the deterministic version of Problem~\prob. Note that the presence of the constraints~(\ref{P_02}) and~(\ref{P_03}) implies that for deciding on $x^t$, the uncertainty for the remaining stages $t+1,\ldots,T$, i.e., the uncertain cost functions $f^{t+1},\ldots,f^T$, should be taken into consideration.

\subsection{Lagrangian duality revisited}
\label{sec_Lagrange}

For formulating the dual problems related to Problem~\prob, we introduce the Lagrange multipliers $\mu := (\mu_j)_{j \in \mathcal{M}}$ and $\lambda := (\lambda_k)_{k \in \mathcal{L}}$ that correspond to Constraints~(\ref{P_02}) and~(\ref{P_03}) respectively. For convenience in later sections, we assume that both $\mu$ and $\lambda$ are row vectors. Furthermore, we denote the concatenation of the two vectors $\mu$ and $\lambda$ by $(\mu,\lambda)$, meaning that $(\mu,\lambda)$ is a vector of dimension $|\mathcal{M}| + |\mathcal{L}|$. Using $\mu$ and $\lambda$, the Lagrangian $L(x^1,\ldots,x^T,\mu,\lambda)$ of Problem~\prob\ is given by
\begin{equation}
L(x^1,\ldots,x^T,\mu,\lambda) := \sum_{t \in \mathcal{T}} f^t (x^t) + \sum_{j \in \mathcal{M}} \mu_j \sum_{t \in \mathcal{T}} g^t_j(x^t) + \sum_{k \in \mathcal{L}} \lambda_k \sum_{t \in \mathcal{T}} h^t_k (x^t).
\label{eq_Lagrangian}
\end{equation}
The corresponding Lagrangian dual function $q(\mu,\lambda)$ is
\begin{equation*}
q(\mu,\lambda) := \min_{x^1,\ldots, x^T} \left(  L(x^1,\ldots,x^T,\mu,\lambda) \ | \  x^t \in \mathcal{C}^t, \ t \in \mathcal{T} \right) .
\end{equation*}
Note that for any Lagrange multiplier vector $(\mu,\lambda)$, the solution to the inner optimization problem of the Lagrangian dual function $q(\mu,\lambda)$ is unique since each cost function~$f^t$ is strictly convex. We denote this unique solution by the vector $x(\mu,\lambda) := (x^t(\mu,\lambda))_{t \in \mathcal{T}}$ and call this vector the \emph{Lagrangian solution} to $q(\mu,\lambda)$.

Throughout this article, we assume that there exists Lagrange multipliers $(\mu^*,\lambda^*)$ such that the optimal solution $x(\mu^*,\lambda^*)$ to the Lagrangian dual function $q(\mu^*,\lambda^*)$ is optimal for the original primal Problem~\prob. We call these multipliers \emph{optimal} for Problem~\prob. In the case of Problem~\prob, the existence of such optimal Lagrange multipliers is equivalent to Problem~\prob\ satisfying strong duality, which can be achieved through a relatively basic constraint qualification such as Slater's condition \cite{Boyd2004}. Moreover, note that the optimal multipliers $(\mu^*,\lambda^*)$ are unique if and only if the optimal solution to Problem~\prob\ satisfies the so-called Linear Independence Constraint Qualification \cite{Wachsmuth2013}. However, to simplify the discussion and without loss of generality, we assume in the derivation of our approach that the optimal multipliers $(\mu^*,\lambda^*)$ are unique.

Many solution approaches for convex optimization problems exploit strong duality, i.e., they iteratively evaluate the Lagrangian dual function for a guess of the Lagrange multipliers and update this guess based on the resulting function value. The reason for this is that often the special structure of the Lagrangian dual function can be exploited to evaluate it efficiently for given Lagrange multipliers $(\mu,\lambda)$. A well-known example of this is the classical convex resource allocation problem (RAP) given by (see also~\cite{Patriksson2008})
\begin{subequations}
\begin{align}
\text{RAP} \ : \ \min_{x \in \mathbb{R}^n} \ & \sum_{i=1}^n f^i (x^i) \label{RAP_01} \\
\text{s.t. } & \sum_{i=1}^n x^i = R, \label{RAP_02} \\
& l^i \leq x^i \leq u^i, \quad i \in \lbrace 1,\ldots,n \rbrace , \label{RAP_03}
\end{align} \label{RAP}%
\end{subequations}
where $R \in \mathbb{R}$ and $l,u \in \mathbb{R}^n$. Given the single Lagrange multiplier $\lambda \in \mathbb{R}$ corresponding to the resource constraint~(\ref{RAP_02}), the Lagrangian dual function $q_{\text{RAP}} (\lambda)$ of RAP is given by
\begin{equation*}
q_{\text{RAP}} (\lambda) := \min_{x \in \mathbb{R}^n} \left( \sum_{i=1}^n f^i (x^i) + \lambda \left(\sum_{i=1}^n x_i - R \right) \ \middle| \  l^i \leq x^i \leq u^i, \ i \in \lbrace ,\ldots,n \rbrace \right) ,
\end{equation*}
and the optimal solution $x (\lambda)$ to the inner optimization problem of $q_{\text{RAP}} (\lambda)$ is given by
\begin{equation}
x^i (\lambda) = \left\{
\begin{array}{ll}
l^i & \text{if } (\nabla f^i)^{-1}(-\lambda) < l^i, \\
(\nabla f^i)^{-1}(-\lambda) & \text{if } l^i \leq (\nabla f^i)^{-1}(-\lambda) \leq u^i, \\
u^i & \text{if } (\nabla f^i )^{-1}(-\lambda) > u^i,
\end{array}
\right. 
\end{equation}
where $\nabla f^i$ denotes the gradient of $f^i$. The state-of-the-art algorithms for solving RAP iteratively update a guess $\hat{\lambda}$ for the optimal multiplier $\lambda^*$ using the optimal solution to $q_{\text{RAP}}(\hat{\lambda})$ and a combination of binary and bisection search on the interval of possible values for $\lambda^*$ (see, e.g., \cite{Patriksson2015}).

For many optimization problems occurring in practice, the Lagrangian is separable in the primal variables. This separability allows us to decompose the optimization problem within the Lagrangian dual function into smaller subproblems whose optimal solutions together form the optimal solution to the whole problem. Since the Lagrangian $L(x^1,\ldots, x^T,\mu,\lambda)$ of Problem~\prob\ is indeed separable in $x^1,\ldots, x^T$, we can do this also for Problem~\prob. More precisely, we can define for each $t \in \mathcal{T}$ the \emph{local} Lagrangian $L^t (x^t,\mu,\lambda)$ and \emph{local} Lagrangian dual function $q^t(x^t,\mu,\lambda)$ by:
\begin{equation*}
L^t (x^t,\mu,\lambda) := f^t (x^t) + \sum_{j \in \mathcal{M}} \mu_j g^t_j (x^t) + \sum_{k \in \mathcal{L}} \lambda_k h^t_k (x^t)
\end{equation*}
and
\begin{align*}
q^t(\mu,\lambda) := \min_{x^t} \left(  L(x^t,\mu,\lambda) \ | \  x^t \in \mathcal{C}^t \right) .
\end{align*}
Note that solving the inner problem of $q(\mu,\lambda)$ is equivalent to solving the inner problems of $q^t(\mu,\lambda)$ for each $t \in \mathcal{T}$. In other words, for each $t \in \mathcal{T}$, the solution to the inner problem of $q^t(\mu,\lambda)$ is equal to the $t^{\text{th}}$ component of the solution to the inner problem of $q(\mu,\lambda)$, namely $x^t(\mu,\lambda)$. We call this solution the \emph{local Lagrangian solution} corresponding to stage~$t$. 

In many applications of Lagrangian duality theory, this separability of the Lagrangian is exploited to allow for parallel computation of the local Lagrangian dual functions $q^t(\mu,\lambda)$ (see, e.g., \cite{Bertsekas1997}). However, we exploit the separability for a different purpose, which we explain in the next subsection.

\subsection{Solution approach}
\label{sec_approach}

The main idea behind our solution approach is as follows. Suppose we are at the start of stage $\bar{t}$ and the corresponding cost function $f^{\bar{t}}$ has been revealed. If we know the optimal Lagrange multipliers $(\mu^*,\lambda^*)$, we can compute the optimal stage decision vector $(x^*)^{\bar{t}}$ by solving the inner problem of $q^{\bar{t}}(\mu^*,\lambda^*)$. This means that we can do this --- and this is the key observation in this article --- without any knowledge of the future cost functions $f^{\bar{t}+1},\ldots,f^T$. As a consequence, the only required data that is uncertain in the computation of the optimal solution vector $(x^*)^{\bar{t}}$ is the vector of optimal Lagrange multipliers $(\mu^*,\lambda^*)$ (recall, that the current cost function~$f^{\bar{t}}$ has already been revealed). Thus, if we have a prediction $(\hat{\mu},\hat{\lambda})$ for $(\mu^*,\lambda^*)$, we can obtain an approximation or \emph{online solution} $\hat{x}^{\bar{t}}$ to the optimal solution vector $(x^*)^{\bar{t}}$ by solving the inner problem of $q^{\bar{t}}(\hat{\mu},\hat{\lambda})$.

%\subsection{Feasibility}

One important question is whether the resulting online solution $\hat{x} := (\hat{x}^t)_{t \in \mathcal{T}}$ is feasible. Note that when the optimal Lagrange multipliers $\mu^*$ and $\lambda^*$ are known, we have the guarantee that the solution to the inner problem of the Lagrangian dual function $q(\mu^*,\lambda^*)$ (or, equivalently, the solution obtained by solving the local Lagrangian dual functions $q^t(\mu^*,\lambda^*)$) is feasible for Constraints~(\ref{P_02}) and~(\ref{P_03}). This is true although these constraints are not enforced explicitly in the Lagrangian dual function. However, when we replace these variables by predictions $\hat{\mu}$ and $\hat{\lambda}$, i.e., when we use the solutions of the inner problems of $q(\hat{\mu},\hat{\lambda})$, we do not have this guarantee anymore. Thus, we must find a different way to ensure that the online solution $\hat{x}^t$ for a given stage~$t \in \mathcal{T}$ is chosen such that there always exists a feasible extension of this solution for the future stages.

We discuss two approaches for this problem. The first approach is to simply add all original constraints to the formulation of the inner problem of $q^{\bar{t}}(\hat{\mu},\hat{\lambda})$, substitute the online solution corresponding to the previous stages in $\lbrace 1,\ldots, \bar{t} - 1 \rbrace$, and solve this adjusted optimization problem instead. This adjusted problem, which we denote by~$(P^{\bar{t}}(\hat{\mu},\hat{\lambda}))$, can be formulated as follows:
\begin{align*}
(P^{\bar{t}}(\hat{\mu},\hat{\lambda})) \ : \ \min_{x^{\bar{t}},\ldots,x^T} \ &  f^{\bar{t}} (x^{\bar{t}}) + \sum_{j \in \mathcal{M}} \hat{\mu}_j g^{\bar{t}}_j (x^{\bar{t}}) + \sum_{k \in \mathcal{L}} \hat{\lambda}_k h^{\bar{t}}_k (x^{\bar{t}}) \\
\text{s.t. } & \sum_{t=1}^{\bar{t}-1} g^t_j (\hat{x}^t) 
+ \sum_{t=\bar{t}}^T g^t_j (x^t) \leq 0, \quad j \in \mathcal{M}, \\
& \sum_{t=1}^{\bar{t}-1} h^t_k (\hat{x}^t) 
+ \sum_{t=\bar{t}}^T h^t_k (x^t) = 0, \quad k \in \mathcal{L}, \\
& x^t \in \mathcal{C}^t, \quad t \in \lbrace \bar{t},\ldots, T \rbrace .
\end{align*}
Observe that this problem is an instance of Problem~\prob\ with horizon $\lbrace \bar{t},\ldots,T \rbrace$ and without the unknown cost functions $f^t$ for $t > \bar{t}$. This suggests that solving this adjusted problem is as easy as solving the deterministic version of the original Problem~\prob. In fact, one could even use the same, possibly tailored, solution methodology, provided this methodology does not require all cost functions to be \emph{strictly} convex.

A second approach, which leads to the same solution but is sometimes easier to execute, is to construct the projection of the feasible set onto the current decision vector and adjust the set $\mathcal{C}^{\bar{t}}$ accordingly. This means that we construct an alternative feasible set $\mathcal{C}^{\bar{t}}_{\text{proj}} \subset \mathbb{R}^{N^{\bar{t}}}$ for $x^{\bar{t}}$ such that choosing $\hat{x}^{\bar{t}}$ from this set ensures that there still exists a feasible solution for the future stages. More precisely, we solve the following problem:
\begin{align*}
(P^{\bar{t}}_{\text{proj}}(\hat{\mu},\hat{\lambda})) \ : \ \min_{x^{\bar{t}}} \ & f^{\bar{t}} (x^{\bar{t}}) + \sum_{j \in \mathcal{M}} \hat{\mu}_j g^{\bar{t}}_j (x^{\bar{t}}) + \sum_{k \in \mathcal{L}} \hat{\lambda}_k h^{\bar{t}}_k (x^{\bar{t}}) \\
\text{s.t. }& x^{\bar{t}} \in \mathcal{C}^{\bar{t}}_{\text{proj}},
\end{align*}
where
\begin{equation*}
\mathcal{C}^{\bar{t}}_{\text{proj}} := \lbrace x^{\bar{t}} \ | \
\text{ there exists } (x^{\bar{t}+1},\ldots,x^T) \text{ such that } (\hat{x}^1,\ldots,\hat{x}^{\bar{t}-1},x^{\bar{t}},\ldots,x^T) \text{ is feasible for P} \rbrace.
\end{equation*}
Note that this projection is not the same as the projection operation that is commonly used in OCO. The latter operation is used to project a candidate solution onto a feasible set (see also, e.g., \cite{Hazan2016}), whereas here we project a feasible set onto a decision vector.

The second approach is worth considering over the first approach if constructing the projection $\mathcal{C}^{\bar{t}}_{\text{proj}}$ is relatively easy. An example of this is the case where all constraints~(\ref{P_02}) and~(\ref{P_03}) are affine and the sets $\mathcal{C}^t$ are polytopes. In this case, the feasible set of Problem~\prob\ is given by a set of linear inequalities and, as a consequence, one can use a method such as Fourier-Motzkin elimination (FME) (see, e.g., \cite{Bertsimas1997}) to compute the projection~$\mathcal{C}^t_{\text{proj}}$. Although FME has an exponential worst-time complexity, several improvements to the original FME algorithm can be made to make it fast in practice (see, e.g., \cite{Bastrakov2015}). In addition, specific problem structures of the considered instance of Problem~\prob\ can be exploited to increase the efficiency of FME. In Section~\ref{sec_battery}, we show that this is indeed possible for a particular battery charging scheduling problem.

Summarizing, using the ODDO-framework we can solve the online version of Problem~\prob\ as follows. First, we compute a prediction $(\hat{\mu},\hat{\lambda})$ of the optimal Lagrange multipliers $(\mu^*,\lambda^*)$ (we discuss several approaches to make such a prediction in Section~\ref{sec_setup}). Subsequently, at the start of each stage $\bar{t} \in \mathcal{T}$, we compute the corresponding online solution vector $\hat{x}^{\bar{t}}$ in one of the following two ways. If the alternative feasible set $\mathcal{C}^{\bar{t}}_{\text{proj}}$ can be computed easily, e.g., by means of FME, then we solve Problem~$(P^{\bar{t}}_{\text{proj}}(\hat{\mu},\hat{\lambda}))$ using this set as input and choose $\hat{x}^{\bar{t}}$ as the solution to this problem. Otherwise, we solve Problem~$(P^{\bar{t}}(\hat{\mu},\hat{\lambda}))$, i.e., we enforce all original constraints into the formulation of the inner problem of the Lagrangian dual function $q^{\bar{t}} (\hat{\mu},\hat{\lambda})$, and take the resulting solution vector corresponding to stage~$\bar{t}$ as the online solution~$\hat{x}^{\bar{t}}$. 

Algorithm~\ref{alg_online} summarizes the ODDO-framework as presented in this section. The check in Line~4 has been included mainly for the sake of completeness: when implementing the algorithm for a specific instance of Problem~\prob\, one can most likely determine on forehand whether or not the projection $\mathcal{C}^t_{\text{proj}}$ is easy to compute and directly go to Line~5 or Line~7.
\begin{algorithm}
\caption{Solving the online version of Problem~\prob\ using the ODDO-framework.}
\label{alg_online}
\begin{algorithmic}[1]
\STATE{\textbf{Input}: prediction $(\hat{\mu},\hat{\lambda})$ of optimal Lagrange multipliers}
\STATE{\textbf{Output}: online solution $\hat{x} := (\hat{x}^1,\ldots,\hat{x}^T)$}
\FOR {$\bar{t}=1,\ldots,T$}
\IF {Projection $\mathcal{C}^{\bar{t}}_{\text{proj}}$ can be computed easily}
\STATE{Compute $\mathcal{C}^{\bar{t}}_{\text{proj}}$ and obtain solution $\tilde{x}^{\bar{t}}$ to Problem~$(P^{\bar{t}}_{\text{proj}}(\hat{\mu},\hat{\lambda}))$}
\ELSE
\STATE {Obtain solution $(\bar{x}^{\bar{t}},\ldots,\tilde{x}^T)$ to Problem $(P^{\bar{t}}(\hat{\mu},\hat{\lambda}))$}
\ENDIF
\STATE{$\hat{x}^{\bar{t}} := \tilde{x}^{\bar{t}}$}
\ENDFOR
\end{algorithmic}
\end{algorithm}

\subsection{An illustrative example}
\label{sec_toy}

As an illustrative example of the ODDO-framework, we consider the following example problem, which is an instance of the resource allocation problem RAP introduced in Section~\ref{sec_Lagrange}:
\begin{subequations}
\begin{align}
\text{E} \ : \ \min_{x^1,x^2,x^3} \ & \sum_{t=1}^3 (x^t)^2 + Y^t x^t \\
\text{s.t. } &x^1 + x^2 + x^3 = 10, \label{toy_02}\\
& 0 \leq x^t \leq 6, \quad t \in \lbrace 1,2,3 \rbrace. \label{toy_03}
\end{align}  %
\end{subequations}
Since all stage vectors $x^t$ are one-dimensional, we omit the subscript index, meaning that in this example $x^t \in \mathbb{R}$ for all $t \in \mathcal{T}$. In this problem, the parameters $Y^1$, $Y^2$, and $Y^3$ are uncertain and are revealed at the start of stages~1,~2, and~3 respectively. Furthermore, there is only a single Lagrange multiplier $\lambda$ that corresponds to the equality constraint~(\ref{toy_02}). For this multiplier, we take as prediction $\hat{\lambda} = 2$.

We start the online optimization process at the beginning of stage~1. At this moment it is revealed that $Y^1 = -4$. Thus, the local Lagrangian dual function at stage~1 is
\begin{equation*}
q^1_{\text{E}}(\lambda) := \min_{x^1} ((x^1)^2 + (-4 + \lambda)x^1 \ | \ 0 \leq x^1 \leq 6 ).
\end{equation*}
Hereby, we still need to take future feasibility into account. Due to the simple structure of Problem~E, we can easily compute the projection of the feasible region onto $x^1$. For this, note that we must have $\hat{x}^1 = 10 - x^2 - x^3$ for some $x^2$ and $x^3$ that satisfy their box constraints~(\ref{toy_03}). It follows that $-2 \leq \hat{x}^1 \leq 10$. Since these bounds are less tight than the original bounds imposed by Constraint~(\ref{toy_03}), we can ignore them and thus we can choose $\hat{x}^1$ as the solution to the inner problem of $q^1_{(E)}(\hat{\lambda})$. This yields $\hat{x}^1 = 1$.

Subsequently, at the beginning of stage~2, it is revealed that $Y^2 = 1$. In order to ensure future feasibility, we must choose $\hat{x}^2$ such that $\hat{x}^2 = 10 - \hat{x}^1 - x^3$ for some $x^3$ that satisfies the bound constraints~(\ref{toy_03}). It follows that we must have $3 \leq \hat{x}^2 \leq 9$. Thus, we must choose $\hat{x}^2$ as the solution to the projected problem $(E^2_{\text{proj}}(2))$, which is given by
\begin{equation*}
\min_{x^2} \ (x^2)^2 + 3x^2 \quad \text{s.t. } 3 \leq x^2 \leq 6.
\end{equation*}
It follows that $\hat{x}^2 = 3$.

Finally, we arrive at stage~3. Assume, that the last uncertain parameter, $Y^3$, is revealed to be $-5$. However, its value has no influence on the online decision $\hat{x}^3$ since it is required to be $\hat{x}^3 = 10 - \hat{x}^1 - \hat{x}^2 = 6$. Note that this does not violate the box constraints~(\ref{toy_03}) for $x^3$. Thus, the online solution obtained by choosing $\hat{\lambda} = 2$ is $\hat{x} = (1,3,6)$ and its objective value is 25. The optimal offline solution, i.e., where the values for $Y^t$ are known on forehand, is given by $x^* = (4,\frac{3}{2},\frac{9}{2})$ with an objective value of $\frac{3}{2}$ and the corresponding optimal Lagrange multiplier is $\lambda^* = -4$.

\section{Motivation and robustness of ODDO}
\label{sec_rob}

With regard to the ODDO-approach presented in Section~\ref{sec_frame}, two important questions arise. The first question is about the predictive value of the optimal Lagrange multipliers: what is the value of knowing the optimal Lagrange multipliers when the cost functions are different from what was expected? In other words, can we say something about the difference between the optimal Lagrange multipliers of two problem instances whose cost functions are (slightly) different? The second question is about the robustness of this approach: how accurate must the predictions of the optimal Lagrange multipliers be such that this approach yields a good approximate online solution? In other words, what can we say about the difference in objective value between the online and optimal offline solution when we use a prediction of the optimal Lagrange multipliers as input for the online approach?

For the first question, we provide in Section~\ref{sec_mot} several arguments that suggest that a small change in the cost functions will lead to only proportionally small changes or even no changes at all in the optimal Lagrange multipliers. In Section~\ref{sec_rob_errors}, we answer the second question by deriving for the special class of optimization problems under submodular constraints a bound on the difference in objective value between the online and offline optimal solution.

\subsection{Predictive value of the optimal Lagrange multipliers}
\label{sec_mot}

To answer the first question, we provide several arguments that suggest that a small change in the cost functions leads to only proportionally small changes or even no changes at all in the optimal Lagrange multipliers. Additionally, we argue that, from a practical point of view, knowledge of the optimal Lagrange multipliers of a given problem provides useful structural information on the problem and its optimal solutions.

The first argument is based on results from sensitivity analysis on parametric convex programs. For this, we assume for the moment that the uncertainty in the cost functions is due to an uncertain parameter $w$ that resides in a given set $\mathcal{W} \subset \mathbb{R}^W$. In other words, each cost function $f^t$ is a function of both $x^t$ and $w$, i.e., $f^t(x^t) \equiv f^t(x^t,w)$. Let $\mathcal{F} \subset \mathbb{R}^N$ denote an arbitrary feasible region for Problem~\prob. For each $w \in \mathcal{W}$, we denote by $I(w)$ the problem instance of \prob\ with feasible set $\mathcal{F}$ and cost functions $f^t(x^t,w)$, $t \in \mathcal{T}$. Moreover, let $(\mu^*(w),\lambda^*(w))$ be the optimal Lagrange multipliers for $I(w)$. It can be shown under a relatively weak constraint qualification, such as the so-called Mangasarian Fromovitz Constraint Qualification, that $(\mu^*(w),\lambda^*(w))$ is Lipschitz continuous in $w$ (see, e.g., \cite{Still2018}). In particular, this means that there exists a positive constant $K_0$ such that for any $w^1,w^2 \in \mathcal{W}$ it holds that
\begin{equation*}
|| (\mu^*(w^1),\lambda^*(w^1)) - (\mu^*(w^2),\lambda^*(w^2))||_2 \leq K_0 ||w^1 - w^2||_2.
\end{equation*}
This implies that the difference in optimal Lagrange multipliers grows only linearly in the change  of the uncertainty parameter $w$.

The second argument is that, for a given feasible set $\mathcal{F}$, there are in general multiple instances of Problem~\prob\ over $\mathcal{F}$, i.e., multiple different cost functions, whose optimal Lagrange multipliers are the same. In other words, the mapping from the set of possible cost functions to the corresponding optimal Lagrange multipliers is not necessary one-to-one. This is because Constraint~(\ref{P_04}) is not dualized in the formulation of the Lagrangian in Equation~(\ref{eq_Lagrangian}). As a consequence, the Lagrange multipliers that correspond to any constraints imposed by the sets $\mathcal{C}^t$ are not specified in a vector of optimal Lagrange multipliers. Thus, it can occur for two problem instances with different cost functions that their Lagrange multipliers differ only in these non-specified Lagrange multipliers. As an example of this, we consider again the convex resource allocation problem RAP introduced in Section~\ref{sec_Lagrange}. Recall that, given the optimal Lagrange multiplier $\lambda^*$, the optimal solution to RAP is given by
\begin{equation*}
x^i (\lambda^*) = \left\{
\begin{array}{ll}
l^i & \text{if } (\nabla f^i)^{-1}(-\lambda^*) < l^i, \\
(\nabla f^i)^{-1}(-\lambda^*) & \text{if } l^i \leq (\nabla f^i)^{-1}(-\lambda^*) \leq u^i, \\
u^i & \text{if } (\nabla f^i )^{-1}(-\lambda^*) > u^i.
\end{array}
\right. 
\end{equation*}

Observe that $\lambda^*$ remains the optimal Lagrange multiplier for this problem for, e.g., any permutation of the index set $\lbrace 1,\ldots,n \rbrace$. Note, however, that the Lagrange multipliers that correspond to the box constraints~(\ref{RAP_03}) would be different. 

%This corresponds with the fact that for $(RAP)$, the Karush-Kuhn-Tucker (KKT) conditions are both sufficient and necessary for optimality and that $(RAP)$ has a unique optimal solution (see, e.g., \cite{Boyd2004}). In general, it can be seen that we have a reduction in the number of values that must be predicted when the dimension of the (data) uncertainty is higher than the number $|\mathcal{L}| + |\mathcal{M}|$ of dualized constraints~(\ref{P_02}) and~(\ref{P_03}).

%The relation partly relies on the difference in dimension between the primal and dual space of the problem. More precisely, if the dimension of the dual space is smaller than the dimension of the primal space, i.e., if $|\mathcal{L}| + |\mathcal{M}| < N$, then there cannot exist a one-to-one mapping from the set of problem instances to their corresponding optimal Lagrange multipliers. On the other hand, when the dimension of the dual space is large, it might be the case that the primal space becomes small since there are more constraints.

The two given arguments together suggest that the optimal Lagrange multipliers may be robust against changes in the cost functions as:
\begin{itemize}
\item
changes in the cost functions does not always lead to a change in the optimal Lagrange multipliers;
\item
even if the optimal Lagrange multipliers change, this change is bounded by the change of the uncertain data that underlies the cost functions.
\end{itemize}

The above discussion indicates that the optimal Lagrange multipliers may provide more structural information about the problem and its optimal solution than the cost functions. Besides the properties discussed in this subsection, there are other properties of the optimal Lagrange multipliers that support this claim. One of these properties is complementary slackness: if for a given inequality constraint~(\ref{P_02}) with index $j \in \mathcal{M}$ the corresponding optimal Lagrange multiplier $\mu^*_j$ is nonzero, then this constraint is tight in the optimal primal solution $x^*$. This is not only a nice theoretical structural relation between the optimal Lagrange multipliers and the optimal primal solution, but can also have a useful physical interpretation depending on the application. Two examples of this are present the two problems that we consider in the evaluation of ODDO in Section~\ref{sec_eval}, namely battery scheduling and inventory management. In the first problem, a tight inequality constraint implies that, in the optimal solution, the battery is either completely full or empty at a given stage (see also Section~\ref{sec_battery}). In the latter problem, tight inequality constraints imply that, in the optimal solution, a given factory produces at a capacity utilization rate of 100\% or that the level of stock in the warehouse has reached either its minimum or maximum allowed level at a given stage (see also Section~\ref{sec_IM}). These insights can not only be used to determine the optimal operation of the battery or the inventory management systems, but also to design these systems and decide on the necessity and profitability of, e.g., investments in battery or warehouse capacity.

\subsection{Robustness of ODDO against prediction errors}
\label{sec_rob_errors}

In this section, we consider the second question concerning the relation between the prediction $(\hat{\mu},\hat{\lambda})$ of the optimal Lagrange multipliers and the difference in objective value between the online and optimal offline solution. To answer this question, we derive a bound on the difference in objective value between the online solution~$\hat{x}(\hat{\mu},\hat{\lambda})$ and the optimal offline solution for the special case of optimization problems with submodular constraints and increasing cost functions. A disadvantage of this bound is that it depends not only on the Lagrange multipliers but also on the uncertain cost functions. Therefore, we refine this bound for the special case where each cost function is of the form $a^t \bar{f}^t (\frac{x^t}{a^t} + b^t)$ where $\bar{f}^t$ is a \emph{known} increasing strictly convex differentiable function, $a := (a^t)_{t \in \mathcal{T}}$ is a \emph{known} positive-valued vector, and $b := (b^t)_{t \in \mathcal{T}}$ is an \emph{uncertain} vector. As shown in \cite{SchootUiterkamp2020a}, in many applications, the cost functions have this form. In particular, this is the case for the battery charging scheduling problem that we consider in Section~\ref{sec_battery}. The resulting refined bound depends only on the Lagrange multipliers and the known functions $\bar{f}^t$ and parameter $a$. 

We first give a precise definition of the problem class that we consider in this section. For this, we require the concepts of submodular functions and base polyhedra. A set function $r \ : \ 2^{\mathcal{T}} \rightarrow \mathbb{R}$ is said to be \emph{submodular} if we have for any subsets $\mathcal{X},\mathcal{Y} \subseteq \mathcal{T}$ that $r(\mathcal{X} \cup \mathcal{Y}) + r(\mathcal{X} \cap \mathcal{Y}) \leq r(\mathcal{X}) + r(\mathcal{Y})$, where we assume without loss of generality that $r(\emptyset) = 0$. The \emph{base polyhedron} associated with a submodular set function $r$ is defined as
\begin{equation*}
\mathcal{B}(r) := \left\{ x \in \mathbb{R}^{T} \ \middle| \ \sum_{t \in \mathcal{X}} x^t \leq r(\mathcal{X}) \ \forall \mathcal{X} \subset \mathcal{T}, \  \sum_{t \in \mathcal{T}} x^t = r(\mathcal{T}) \right\} .
\end{equation*}
A well-known property of base polyhedra is that for any vector $x \in \mathcal{B}(r)$ the set of tight inequalities that define $\mathcal{B}(r)$ is closed under union and intersection:
\begin{lemma}
If for a submodular function $r : 2^{\mathcal{T}} \rightarrow \mathbb{R}$, vector $x \in \mathcal{B}(r)$, and two sets $\mathcal{X},\mathcal{Y} \subseteq \mathcal{T}$ we have that $\sum_{t \in \mathcal{X}} x^t = r(\mathcal{X})$ and $\sum_{t \in \mathcal{Y}} x^t = r(\mathcal{Y})$, then we also have that $\sum_{t \in \mathcal{X} \cup \mathcal{Y}} x^t = r(\mathcal{X} \cup \mathcal{Y})$ and $\sum_{t \in \mathcal{X} \cap \mathcal{Y}} x^t = r(\mathcal{X} \cap \mathcal{Y})$.
\label{lemma_sub}
\end{lemma}
\begin{proof}
Since $\sum_{t \in \mathcal{X}} x^t = r(\mathcal{X})$ and $\sum_{t \in \mathcal{Y}} x^t = r(\mathcal{Y})$, it follows from the submodularity of $r$ that
\begin{equation*}
\sum_{t \in \mathcal{X}} x^t + \sum_{t \in \mathcal{Y}} x^t = r(\mathcal{X}) + r(\mathcal{Y})
\geq
r(\mathcal{X} \cup \mathcal{Y}) + r(\mathcal{X} \cap \mathcal{Y})
\geq
\sum_{t \in \mathcal{X} \cup \mathcal{Y}} x^t + \sum_{t \in \mathcal{X} \cap \mathcal{Y}} x^t
= \sum_{t \in \mathcal{X}} x^t + \sum_{t \in \mathcal{Y}} x^t .
\end{equation*}
Thus, $\sum_{t \in \mathcal{X} \cup \mathcal{Y}} x^t + \sum_{t \in \mathcal{X} \cap \mathcal{Y}} x^t = r(\mathcal{X} \cup \mathcal{Y}) + r(\mathcal{X} \cap \mathcal{Y})$, which directly implies the result of the lemma.
\end{proof}
\noindent For a thorough treatment of submodular functions and base polyhedra, we refer to \cite{Fujishige2005}.

We now define the class $I_{\text{sub}}$ of problem instances that we consider in this section:
\begin{definition}
The class $I_{\text{sub}}$ consists of all problem instances of the form
\begin{align*}
\min_{x^1,\ldots,x^T} \ & \sum_{t \in \mathcal{T}} f^t(x^t) \\
\text{s.t. } & x \in \mathcal{B}(r),
\end{align*}
where for all $t \in \mathcal{T}$ the vector $x^t$ is one-dimensional (i.e., $N^t = 1$), the function $f^t$ is increasing, strictly convex, and continuously differentiable, and $r$ is a submodular set function on $\mathcal{T}$.
\end{definition}
\noindent Note that each instance in $I_{\text{sub}}$ is also an instance of Problem~\prob\ by choosing
\begin{align*}
\mathcal{M} & := \lbrace \mathcal{X} \ | \ \mathcal{X} \in 2^{\mathcal{T}}, \ 1 < |\mathcal{X}| < T \rbrace, \\
\mathcal{L} &:= \lbrace \mathcal{T} \rbrace, \\
\mathcal{C}^t & := [r(\mathcal{T}) - r(\mathcal{T} \backslash \lbrace t \rbrace ) , r(\lbrace t \rbrace)], \\
g_{\mathcal{X}}^t (x^t) & := \begin{cases} 
x^t - \frac{r(\mathcal{X})}{|\mathcal{X}|} & \text{if } t \in \mathcal{X}, \\ 0 & \text{otherwise,} \end{cases} \\
h_{\mathcal{T}}^t (x^t) & := x^t -  \frac{r(\mathcal{T})}{T}, \quad t \in \mathcal{T}.
\end{align*}
Thus, for a given $t \in \mathcal{T}$, we can write the local Lagrangian dual function $q_{sub}^t(\mu,\lambda)$ as
\begin{equation*}
q_{\text{sub}}^t(\mu,\lambda) := \min_{x^t} \left( f^t (x^t) + \left( \sum_{\mathcal{X} \in \mathcal{M}: \ \mathcal{X} \ni t} \mu_{\mathcal{X}} + \lambda_{\mathcal{T}} \right) x^t \ \middle| \ r(\mathcal{T}) - r(\mathcal{T} \backslash \lbrace t \rbrace ) \leq x^t \leq  r(\lbrace t \rbrace) \right) ,
\end{equation*}
and, by letting $\nu^t(\mu,\lambda) := -\sum_{\mathcal{X} \in \mathcal{M}: \ \mathcal{X} \ni t} \mu_{\mathcal{X}} - \lambda_{\mathcal{T}}$, the local Lagrangian solution is given by
\begin{equation}
x^t(\mu,\lambda) := \begin{cases}
r(\mathcal{T}) - r(\mathcal{T} \backslash \lbrace t \rbrace ) & \text{if } (\nabla f^t)^{-1}(\nu^t(\mu,\lambda)) < r(\mathcal{T}) - r(\mathcal{T} \backslash \lbrace t \rbrace ), \\
(\nabla f^t)^{-1}(\nu^t(\mu,\lambda)) & \text{if } r(\mathcal{T}) - r(\mathcal{T} \backslash \lbrace t \rbrace ) \leq (\nabla f^t)^{-1}(\nu^t(\mu,\lambda)) \leq r(\lbrace t \rbrace), \\
r(\lbrace t \rbrace) & \text{if } (\nabla f^t)^{-1}(\nu^t(\mu,\lambda)) > r(\lbrace t \rbrace).
\end{cases}
\label{eq_local_L_sub}
\end{equation}

For this specific instance class $I_{\text{sub}}$, we can bound the difference in objective value between the online and optimal offline solution when the multiplier prediction $(\hat{\mu},\hat{\lambda})$ is an ``under-prediction'' of the optimal Lagrange multipliers $(\mu^*,\lambda^*)$, i.e., the inequality $(\hat{\mu},\hat{\lambda}) \leq (\mu^*,\lambda^*)$ holds component-wise. Apart from the predicted and optimal Lagrange multipliers $(\hat{\mu},\hat{\lambda})$ and $(\mu^*,\lambda^*)$, this bound depends on the cost functions $f^t$ and the inverses of their gradients $(\nabla f^t)^{-1}$. More precisely, the bound is the difference in objective value between the (not necessarily feasible) solutions $((\nabla f^t)^{-1}(v^t(\hat{\mu},\hat{\lambda})))_{t \in \mathcal{T}}$ and $((\nabla f^t)^{-1}(v^t(\mu^*,\lambda^*)))_{t \in \mathcal{T}}$. These solutions are closely related to the Lagrangian solutions $x(\hat{\mu},\hat{\lambda})$ and $x(\mu^*,\lambda^*)$ as given by Equation~(\ref{eq_local_L_sub}). We state this bound by means of Theorem~\ref{th_main}.

 \begin{theorem}
 For any problem instance in $I_{\text{sub}}$ and a given prediction $(\hat{\mu},\hat{\lambda})$ of the optimal Lagrange multipliers $(\mu^*,\lambda^*)$ such that $(\hat{\mu},\hat{\lambda}) \leq (\mu^*,\lambda^*)$, it holds that
\begin{equation*}
 \sum_{t \in \mathcal{T}} f^t (\hat{x}^t(\hat{\mu},\hat{\lambda})) - \sum_{t \in \mathcal{T}} f^t (\hat{x}^t(\mu^*,\lambda^*))  \leq 
 \sum_{t \in \mathcal{T}} f^t((\nabla f^t)^{-1}(\nu^t(\hat{\mu},\hat{\lambda})))
 -  \sum_{t \in \mathcal{T}} f^t((\nabla f^t)^{-1}(\nu^t(\mu^*,\lambda^*))).
 \end{equation*}
\label{th_main}
 \end{theorem}
We prove Theorem~\ref{th_main} using Lemmas~\ref{lemma_instance}-\ref{lemma_prop_2}. Lemma~\ref{lemma_instance} states that the local Lagrangian solution $x(\mu,\lambda)$ is non-increasing in $(\mu,\lambda)$. This implies that the behavior of the initial online solution $x^t(\mu,\lambda)$ can, to some extend, be controlled via the predicted Lagrange multipliers. Lemma~\ref{lemma_tight} implies that whenever the local Lagrangian solution $x^t(\hat{\mu},\hat{\lambda})$ needs to be reduced in order to preserve future feasibility, at least one constraint involving the current stage~$t$ but no future stages is tight in the online solution. This means that this constraint is fulfilled ``earlier'' than it would have been when using the optimal Lagrange multipliers as input. This suggests that the constraint is (nearly) tight in the optimal solution. Finally, Lemma~\ref{lemma_prop_2} uses Lemmas~\ref{lemma_instance} and~\ref{lemma_tight} to prove that, in the case of an under-prediction of the Lagrange-multipliers, the cost of an online solution $\hat{x}^t(\mu,\lambda)$ will always be less than or equal to the cost of the corresponding local Lagrangian solution $x^t(\mu,\lambda)$, even when adjustments have been done in order to maintain feasibility. %The proofs of these three lemmas can be found in Appendices~\ref{app_lemma_instance}-\ref{app_lemma_prop_2} respectively.

 %Property~\ref{prop_online} is required for the main robustness result, which we present shortly, whereas Property~\ref{prop_mono} is required to show that Property~\ref{prop_online} holds for a specific problem class that we introduce later in this section.

\begin{lemma}
For any instance in $I_{\text{sub}}$, the local Lagrangian solutions $x^t(\mu,\lambda)$ are nonincreasing in $(\mu,\lambda)$, i.e., for any two vectors of Lagrange multipliers $(\underline{\mu},\underline{\lambda})$ and $(\bar{\mu},\bar{\lambda})$ such that $(\underline{\mu},\underline{\lambda}) \leq (\bar{\mu},\bar{\lambda})$, it holds that $x^t (\underline{\mu},\underline{\lambda}) \geq x^t (\bar{\mu},\bar{\lambda})$ for all $t \in \mathcal{T}$.
\label{lemma_instance}
\end{lemma}
\begin{proof}
Since $f^t$ is strictly convex and continuously differentiable, its gradient $\nabla f^t$ is continuous and monotonically increasing. It follows from the inverse function theorem that also the inverse $(\nabla f^t)^{-1}$ is continuous and monotonically increasing (see, e.g., \cite{Spivak1965}). Thus, for any two vectors of multipliers $(\underline{\mu},\underline{\lambda})$ and $(\bar{\mu},\bar{\lambda})$ such that  $(\underline{\mu},\underline{\lambda}) \leq (\bar{\mu},\bar{\lambda})$, it holds that 
\begin{align*}
(\nabla f^t)^{-1} (v^t(\underline{\mu},\underline{\lambda})) &= 
(\nabla f^t)^{-1} \left( -\sum_{\mathcal{X} \in \mathcal{M}: \ \mathcal{X} \ni t} \underline{\mu}_{\mathcal{X}} 
- \underline{\lambda}_{\mathcal{T}} \right) \\
& \geq
(\nabla f^t)^{-1} \left( -\sum_{\mathcal{X} \in \mathcal{M}: \ \mathcal{X} \ni t} \bar{\mu}_{\mathcal{X}} 
- \bar{\lambda}_{\mathcal{T}} \right)
=
(\nabla f^t)^{-1} (v^t(\bar{\mu},\bar{\lambda})) 
.
\end{align*}
The expression in Equation~(\ref{eq_local_L_sub}) for the local Lagrangian solution $x^t(\mu,\lambda)$ implies that $x^t(\mu,\lambda)$ can be seen as a piecewise nondecreasing function of $(\nabla f^t)^{-1}(v^t(\mu,\lambda))$. Thus, it follows that $x^t(\underline{\mu},\underline{\lambda}) \geq x^t(\bar{\mu},\bar{\lambda})$, which proves the lemma.
\end{proof}

\begin{lemma}
For any instance in $I_{\text{sub}}$, a given $\bar{t} \in \mathcal{T}$, and given multipliers $(\mu,\lambda)$ and corresponding online solution $\hat{x}^{\bar{t}}(\mu,\lambda)$ and local Lagrangian solution $x^{\bar{t}}(\mu,\lambda)$, we have:
\begin{equation*}
\hat{x}^{\bar{t}}(\mu,\lambda) < x^{\bar{t}}(\mu,\lambda)
\Longrightarrow 
\exists \bar{\mathcal{S}}^{\bar{t}} \subseteq \left\{ 1,\ldots,\bar{t} \right\} \text{ with } \bar{\mathcal{S}}^{\bar{t}} \ni \bar{t} \text{ such that } \sum_{t \in \bar{\mathcal{S}}^{\bar{t}}} \hat{x}^t(\mu,\lambda) = r\left(\bar{\mathcal{S}}^{\bar{t}} \right).
\end{equation*}
\label{lemma_tight}
\end{lemma}
\begin{proof}
See Appendix~\ref{app_lemma_tight}.
\end{proof}

\begin{lemma}
For any instance in $I_{\text{sub}}$ and given any Lagrange multipliers $(\mu,\lambda)$ such that $(\mu,\lambda) \leq (\mu^*,\lambda^*)$, it holds that $\hat{x}^t (\mu,\lambda) \leq x^t(\mu^*,\lambda^*)$ for all $t \in \mathcal{T}$.
\label{lemma_prop_2}
\end{lemma}
\begin{proof}
See Appendix~\ref{app_lemma_prop_2}.
\end{proof}

Using this last lemma, we now prove Theorem~\ref{th_main}:
 \begin{proof}[Proof of Theorem~\ref{th_main}]
 We prove the theorem by showing that the inequality stated in the theorem holds for each $t \in \mathcal{T}$ individually, i.e., $ f^t(\hat{x}^t(\hat{\mu},\hat{\lambda})) - f^t(x^t(\mu^*,\lambda^*)) \leq 
  f^t((\nabla f^t)^{-1}(\nu^t(\hat{\mu},\hat{\lambda})))
 -  
  f^t((\nabla f^t)^{-1}(\nu^t(\mu^*,\lambda^*)))$. The result then follows by summing this inequality over $t$.
  
   Without loss of generality, we may assume that the local Lagrangian solutions $x^t(\hat{\mu},\hat{\lambda}) \neq x^t(\mu^*,\lambda^*)$ and that $\nu^t(\hat{\mu},\hat{\lambda}) \neq \nu^t(\mu^*,\lambda^*)$. Thus, by Lemma~\ref{lemma_prop_2} and the fact that $v^t(\mu,\lambda)$ is non-increasing in $(\mu,\lambda)$, it follows that $x^t(\hat{\mu},\hat{\lambda}) < x^t(\mu^*,\lambda^*)$ and $\nu^t(\hat{\mu},\hat{\lambda}) > \nu^t(\mu^*,\lambda^*)$ respectively. By the definition of the local Lagrangian solution $x^t(\mu,\lambda)$ in Equation~(\ref{eq_local_L_sub}), this implies that $x^t(\hat{\mu},\hat{\lambda}) \leq (\nabla f^t)^{-1}(v(\hat{\mu},\hat{\lambda}))$ and $x^t(\mu^*,\lambda^*) \geq (\nabla f^t)^{-1}( v(\mu^*,\lambda^*))$.
 Thus, for each $t \in \mathcal{T}$, it follows by Lemma~\ref{lemma_prop_2} and the fact that $f^t$ is increasing that
 \begin{align*}
 f^t(\hat{x}^t(\hat{\mu},\hat{\lambda})) - f^t(x^t(\mu^*,\lambda^*))
& \leq 
  f^t(x^t(\hat{\mu},\hat{\lambda})) - f^t(x^t(\mu^*,\lambda^*)) \\
 & \leq
 f^t((\nabla f^t)^{-1}(\nu^t(\hat{\mu},\hat{\lambda})))
 -  
  f^t((\nabla f^t)^{-1}(\nu^t(\mu^*,\lambda^*))).
 \end{align*}
\end{proof}
%Summing the above expression over $t$ yields the result of the theorem.

Theorem~\ref{th_main} provides us with a bound on the difference in objective value between the online and optimal offline solution to Problem~\prob. Note, however, that the bound in Theorem~\ref{th_main} depends on the cost functions $f^t$ and the inverses of their gradients $(\nabla f^t)^{-1}$ that are both uncertain. In the following, for a particular type of cost functions, we refine the bound in Theorem~\ref{th_main} so that it depends only on the Lagrange multipliers and some other known structures. For this, suppose that each cost function $f^t$ is given by $a^t \bar{f}^t(\frac{x^t}{a^t} + b^t)$, where the increasing strictly convex differentiable function $\bar{f}^t$ and the parameter $a^t \in \mathbb{R}_{>0}$ are known but the parameter $b^t \in \mathbb{R}$ is uncertain. We show in Corollary~\ref{col_main_sep} that for instances in $I_{\text{sub}}$ whose cost functions have this structure, the bound in Theorem~\ref{th_main} does not depend on any uncertain functions or parameters. Subsequently, as an example of how the bound can be simplified further for specific choices of $\bar{f}^t$, we apply this result to the cases where each function $\bar{f}^t$ is a power function or an exponential function in Corollaries~\ref{col_bound_poly} and~\ref{col_bound_exp} respectively. The former case applies to the battery charging scheduling problem considered in Section~\ref{sec_battery}.
\begin{corollary}
If $f^t(x^t) := a^t \bar{f}^t(\frac{x^t}{a^t} + b^t)$ for a given increasing strictly convex differentiable function $\bar{f}^t$ and a parameter $a^t \in \mathbb{R}_{>0}$ but uncertain parameter $b^t \in \mathbb{R}$, the bound in Theorem~\ref{th_main} becomes
\begin{equation*}
 \sum_{t \in \mathcal{T}} a^t \bar{f}^t((\nabla \bar{f}^t)^{-1}(\nu^t(\hat{\mu},\hat{\lambda})))
 -  \sum_{t \in \mathcal{T}} a^t \bar{f}^t((\nabla \bar{f}^t)^{-1}(\nu^t(\mu^*,\lambda^*))).
 \end{equation*}
 \label{col_main_sep}
\end{corollary}
\begin{proof}
Since $\nabla f^t(x^t) = \nabla \bar{f}^t (\frac{x^t}{a^t} + b^t)$ and $(\nabla f^t)^{-1}(\delta) = a^t((\nabla \bar{f}^t)^{-1}(\delta) - b^t)$ for any $\delta$ in the range of $\nabla f^t$, we have that
\begin{equation*}
f^t((\nabla f^t)^{-1}(\delta)) = f^t(a^t((\nabla \bar{f}^t)^{-1}(\delta) - b^t)) = a^t \bar{f}^t ((\nabla \bar{f}^t)^{-1}(\delta)).
\end{equation*}
The result follows by substituting this expression in the bound of Theorem~\ref{th_main} for $\delta \in \lbrace v^t(\hat{\mu},\hat{\lambda}), v^t(\mu^*,\lambda^*) \rbrace$.
\end{proof}

%For specific choices of $\bar{f}^t$, the bound in Corollary~\ref{col_main_sep} can be simplified even further. We state these simplifications here for the cases where each $\bar{f}^t$ is a power function or an exponential function.
\begin{corollary}
If $\bar{f}^t(y) = K y^c$ for some $K \in \mathbb{R}_{>0}$ and $c \geq 1$ for all $t \in \mathcal{T}$, the bound in Theorem~\ref{th_main} becomes
\begin{equation*}
K \left(\frac{1}{cK}\right)^{\frac{c}{c-1}} 
\left( \sum_{t \in \mathcal{T}} a^t (\nu^t(\hat{\mu},\hat{\lambda}))^{\frac{c}{c-1}}
-  \sum_{t \in \mathcal{T}} a^t (\nu^t(\mu^*,\lambda^*))^{\frac{c}{c-1}} \right).
\end{equation*}
\label{col_bound_poly}
\end{corollary}
\begin{proof}
Follows by substituting $(\nabla \bar{f}^t)^{-1}(\delta) = \left( \frac{\delta}{cK}\right)^{\frac{1}{c-1}}$ for $\delta \in \lbrace v^t(\hat{\mu},\hat{\lambda}), v^t(\mu^*,\lambda^*) \rbrace$ in the refined bound of Corollary~\ref{col_main_sep}.
\end{proof}
 \begin{corollary}
 If $\bar{f}^t(y) = K e^y$ for some $K \in \mathbb{R}_{>0}$ for all $t \in \mathcal{T}$, the bound in Theorem~\ref{th_main} becomes
 \begin{equation*}
 \sum_{t \in \mathcal{T}} a^t \nu^t(\hat{\mu},\hat{\lambda})
- \sum_{t \in \mathcal{T}} a^t \nu^t(\mu^*,\lambda^*)
 \end{equation*}
 \label{col_bound_exp}
 \end{corollary}
 \begin{proof}
 Follows by substituting $(\nabla \bar{f}^t)^{-1}(\delta) = \ln \frac{\delta}{K}$ for $\delta \in \lbrace v^t(\hat{\mu},\hat{\lambda}), v^t(\mu^*,\lambda^*) \rbrace$ in the refined bound of Corollary~\ref{col_main_sep}.
 \end{proof}
 
An alternative bound can be deduced using the fact that both $f^t$ as a function of $x^t$ and $x^t(\mu,\lambda)$ as a function of $(\mu,\lambda)$ are Lipschitz continuous (see, e.g., \cite{Robinson1980,Shvartsman2012}). This yields a bound of the form $K ||(\hat{\mu},\hat{\lambda}) - (\mu^*,\lambda^*)||_2$ where $K$ is a constant that depends on the Lipschitz constants of each function $f^t$ and $x^t(\mu,\lambda)$. This bound suggests that the difference in objective values grows linearly in the 2-norm of the difference between the predicted and optimal Lagrange multipliers. However, this bound depends on the uncertain cost functions via the constant~$K$. Moreover, the linear behavior of the bound is inconsistent with the nature of the gradients for cost functions that are not quadratic. In contrast, our bounds in Theorem~\ref{th_main} and Corollaries~\ref{col_main_sep}-\ref{col_bound_exp} actively incorporate the (gradients of) the specific cost function structure itself.

We conclude this section with a note on whether the result of Theorem~\ref{th_main} could be extended also to instances that do not fall within the class $I_{\text{sub}}$. We observe that Lemma~\ref{lemma_instance} can be extended to broader classes of problems with, e.g., multi-dimensional stage vectors, nonlinear and non-submodular inequality constraints, and nonseparable objectives (see, e.g., \cite{SchootUiterkamp2019a}). This leads us to the question whether one can also extend Lemmas~\ref{lemma_tight} and~\ref{lemma_prop_2} to broader classes of problems. Since the proofs of Lemmas~\ref{lemma_tight} and~\ref{lemma_prop_2} rely heavily on the submodular constraint structure (as opposed to the proof of Lemma~\ref{lemma_instance}), it seems unlikely that these lemmas can be extended to problem instances with non-submodular constraints. However, the incorporation of other extensions such as multi-dimensional stage vectors and nonseparable objectives is not limited by the submodular constraint structure and it may thus be very well possible to extend the result to such problem classes.

\section{Evaluation}
\label{sec_eval}

In this section, we evaluate the performance of the ODDO-framework on two types of optimization problems. We introduce and formulate these problems in Sections~\ref{sec_battery} and~\ref{sec_IM} respectively. Subsequently, in Section~\ref{sec_setup}, we explain the goal and setup of our evaluation, including our approach for predicting the optimal Lagrange multipliers. Finally, in Section~\ref{sec_results}, we present and discuss the results of our evaluations.

\subsection{Battery scheduling}
\label{sec_battery}
Batteries play a crucial role in current and future envisioned micro-grids. On the one hand, they offer the flexibility to store a surplus of produced renewable energy from, e.g., solar panels. On the other hand, they can provide energy to consumers in the micro-grid when the currently produced renewable energy is insufficient to meet the demand. In this way, they help to minimize the amount of energy that has to be exchanged with the main distribution grid and thus help to maximize self-consumption.

We consider a micro-grid with a single battery. The battery scheduling problem can be formulated as follows. We schedule the charging and discharging of the battery over a finite horizon of $T$ equidistant time intervals, each of length~$\Delta t$. At the beginning of each time interval~$t$, we have to decide at which rate~$x^t$ the battery will charge during this interval. This rate $x^t$ is restricted by the minimum and maximum charging rates~$l^t$ and $u^t$ of the battery. Furthermore, the amount of energy that can be stored within the battery, the state-of-charge (SoC), cannot exceed the minimum and maximum capacity of the battery. We denote for interval~$t \in \mathcal{T} = \lbrace 1,\ldots,T \rbrace$ this minimum and maximum allowed SoC by the constants $\underline{C}^t$ and $\bar{C}^t$ respectively. Finally, since the operation of the battery does not stop at the end of the scheduling horizon, we specify a desired SoC at the end of this horizon.We express this by setting $\underline{C}^T = \bar{C}^T = C$ for some constant $C$. 

The goal in the battery scheduling problem is to charge the battery such that the interaction of the local grid with the main grid is minimized. Additionally, we aim to distribute the remaining exchange equally over time to minimize peak consumption, which in turn reduces the energy losses in the system and the stress put on grid assets such as transformers. To model this, we choose as objective function the sum of squares of the net exchange with the main grid. This leads to the objective function $\sum_{t \in \mathcal{T}} (p^t + x^t)^2$, where $p^t$ denotes the net consumption within the micro-grid excluding the battery, i.e., the energy consumption minus the energy production, during time interval~$t$. This net consumption is the uncertain parameter in the scheduling problem as we do not know the vector $p := (p^t)_{t \in \mathcal{T}}$ in advance, but instead only learn each value $p^t$ at the start of the corresponding interval~$t$ by, e.g., measuring the net consumption at the transformer. Summarizing, we consider the following optimization problem:
\begin{subequations}
\begin{align}
\text{BATTERY} \ : \ \min_{x^1,\ldots,x^T} \ & \sum_{t \in \mathcal{T}} (p^t + x^t)^2 \\
\text{s.t. } & \underline{C}^t \leq \Delta t \sum_{s=1}^t x^s \leq \bar{C}^t, \quad t \in \mathcal{T} \backslash \lbrace T \rbrace, \label{battery_02}\\
& \Delta t \sum_{t \in \mathcal{T}} x^t = C, \label{battery_03} \\
& l^t \leq x^t \leq u^t, \quad t \in \mathcal{T}. \label{battery_04}
\end{align} \label{battery}%
\end{subequations}

We note that BATTERY belongs to the problem class $I_{\text{sub}}$ defined in Section~\ref{sec_rob_errors} when $p^t + l^t \geq 0$ for each $t \in \mathcal{T}$. This is because the lower bound inequality constraints~(\ref{battery_02}) can be rewritten to upper bound inequality constraints using Constraint~(\ref{battery_03}). The sets corresponding to these and the original upper bound inequality constraints~(\ref{battery_02}) form a so-called cross-free family, which directly implies that the feasible set defined by these constraints is a base polyhedron (see \cite{Fujishige1984,Fujishige2005}). Also, note that the bound in Corollary~\ref{col_bound_poly} applies to this problem by choosing $K = 1$, $c=2$, $a^t = 1$, and $b^t = p^t$ for all $t \in \mathcal{T}$.

For each $t \in \mathcal{T} \backslash \lbrace T \rbrace$, let $\mu^-_{t}$ and $\mu^+_{t}$ denote the Lagrange multipliers for the lower bound and upper bound constraint in (\ref{battery_02}) respectively. The Lagrangian dual function $q_{\text{BATTERY}}(\mu,\lambda)$ of Problem~BATTERY can be written as
\begin{align*}
q_{\text{BATTERY}}(\mu,\lambda) := \min_{x^1,\ldots,x^T} 
\left( \begin{array}{l}
\sum_{t \in \mathcal{T}} (p^t + x^t)^2 + \sum_{t=1}^{T-1} \mu^+_{t} \left(\sum_{s=1}^t x^s - \frac{\bar{C}^t}{\Delta t} \right) \\
\quad + \sum_{t=1}^{T-1} \mu^-_{t} \left(\frac{\underline{C}^t}{\Delta t} - \sum_{s=1}^t x^s \right) + \lambda \sum_{t \in \mathcal{T}} x^t \\
\text{s.t. } l^t \leq x^t \leq u^t, \quad t \in \mathcal{T} \end{array} \right).
\end{align*}
Thus, the local Lagrangian dual function $q^t_{\text{BATTERY}}(\mu,\lambda)$ can be written for $t \in \mathcal{T} \backslash \lbrace T \rbrace$ as
\begin{equation*}
q^t_{\text{BATTERY}}(\mu,\lambda) := \min_{x^t} 
\left( 
 (p^t + x^t)^2 + \sum_{s=t}^{T-1} (\mu^+_{s} - \mu^-_{s}) x^t + \lambda x^t \ \middle| \ 
 l^t \leq x^t \leq u^t \right)
\end{equation*}
and for $t = T$ as
\begin{equation*}
q^T_{\text{BATTERY}}(\mu,\lambda) := \min_{x^T} 
\left( 
 (p^T + x^T)^2 + \lambda x^T \ \middle| \ 
 l^T \leq x^T \leq u^T \right) .
\end{equation*}

We mentioned in Section~\ref{sec_approach} that, for some cases, the structure of the constraints can be exploited to compute the projection set $\mathcal{C}^t_{\text{proj}}$ efficiently. This is also the case for Problem~BATTERY. More precisely, applying FME to these constraints yields the projection for all variables in $O(T)$ arithmetical operations that can be done at the start of the first time interval (see Appendix~\ref{app_battery}). As a consequence, we do not need to solve at the start of each time interval an entire instance of Problem~BATTERY but can instead use the pre-calculated sets $\mathcal{C}^t_{\text{proj}}$. This saves a lot of computation time, which is crucial for using the algorithm on embedded systems in energy management applications with low computational power.

In our evaluation, we consider the problem of scheduling the battery over one day divided into 15-minute time intervals (i.e., $T = 96$). This interval length is chosen with regard to contract durations within several European energy market (see, e.g., \cite{Markle-Hus2018}). Furthermore, we assume that the initial and desired SoC is set to 50\% of the capacity. For the uncertain consumption $p$, we use real consumption data of 72 households that were obtained in a field test in the Dutch town of Heeten \cite{Reijnders2018}. We base the battery parameters on the setting in the same field test. More precisely, we choose $l^t = -8.67 \cdot 10^3$ and $u^t = 8.67 \cdot 10^3$ for all $t \in \mathcal{T}$, $\underline{C}^t = -5.89 \cdot 10^3$ and $\bar{C}^t = 5.89 \cdot 10^3$ for all $t < T$, $\underline{C}^T = \bar{C}^T = 0$, and $\Delta t = \frac{1}{4}$. We note that, given the used consumption profile $p$, the resulting problem instance satisfies the assumptions of Theorem~\ref{th_main}.

\subsection{Inventory management}
\label{sec_IM}

As second problem, we consider a basic inventory management problem that was considered in \cite{Ben-Tal2004} to evaluate the performance of adjustable RO for linear programs. This model is a single product inventory system that consists of a warehouse and multiple factories. The goal is to minimize the production cost over all factories while satisfying the required demand of the product for each stage.

We denote the set of factories by $\mathcal{N}$ and for each stage~$t \in \mathcal{T}$ and factory~$i \in \mathcal{N}$, the variable $x^t_i$ denotes the amount of (divisible) product that factory~$i$ produces during stage~$t$. Moreover, we denote by $c^t_i$ the cost of producing one unit of the product at factory~$i$ during stage~$t$. For each stage~$t \in \mathcal{T}$, there is a demand~$d^t$ for the product that must be satisfied. At the start of the first stage, an amount of $S$ is already present in the warehouse. Furthermore, the minimum required level of stock in the warehouse is $L$ and the maximum capacity is $U$. Finally, each factory~$i \in \mathcal{N}$ has a maximum production capacity of $u^t_i$ for stage~$t \in \mathcal{T}$ and a maximum total production capacity of $C_i$ for the entire horizon, i.e., $\sum_{t \in \mathcal{T}} x^t_i \leq C_i$. Summarizing, we get the following optimization problem:

\begin{subequations}
\begin{align}
\text{IM} \ : \ \min_{x^1,\ldots, x^T} \ & \sum_{t \in \mathcal{T}} \sum_{i \in \mathcal{N}} c^t_i x^t_i \\
\text{s.t. } & \sum_{t \in \mathcal{T}} x^t_i \leq C_i, \quad i \in \mathcal{N}, \label{IM_02} \\
& L \leq S + \sum_{s=1}^t \sum_{i \in \mathcal{N}} x^s_i - \sum_{s=1}^t d^s \leq U, \quad t \in \mathcal{T}, \label{IM_03} \\
& 0 \leq x^t_i \leq u^t_n, \quad t \in \mathcal{T}, \ i \in \mathcal{N}.
\end{align} %
\end{subequations}
In \cite{Ben-Tal2004}, the demand was taken as the uncertain parameter. Here, we instead assume the demand to be certain and the production costs to be uncertain. 

Let $\mu' := (\mu'_i)_{i \in \mathcal{N}}$, $\mu^- := (\mu^-_{t})_{t \in \mathcal{T}}$, and $\mu^+ := (\mu^+_{t})_{t \in \mathcal{T}}$ denote the Lagrange multipliers corresponding to Constraints~(\ref{IM_02}) and the lower and upper bound Constraints~(\ref{IM_03}) respectively. The Lagrangian dual function $q_{\text{IM}}(\mu)$ of Problem~IM can be written as
\begin{equation*}
q_{\text{IM}}(\mu) := \min_{x^1,\ldots,x^T} \ \left( \begin{array}{l}
\sum_{t \in \mathcal{T}} \sum_{i \in \mathcal{N}} c^t_i x^t_i + \sum_{i \in \mathcal{N}} \mu'_i \left( \sum_{t \in \mathcal{T}} x^t_i - C_i \right) \\
\quad
+ \sum_{t \in \mathcal{T}} \mu^+_{t} \left( S + \sum_{s=1}^t \sum_{i \in \mathcal{N}} x^s_i - \sum_{s=1}^t d^s - U \right) \\
\quad
+ \sum_{t \in \mathcal{T}} \mu^-_{t} \left( - S - \sum_{s=1}^t \sum_{i \in \mathcal{N}} x^s_i + \sum_{s=1}^t d^s + L  \right) \\
\text{s.t. } 0 \leq x^t_i \leq u^t_i, \quad t\in \mathcal{T}, i \in \mathcal{N}.
\end{array}
\right),
\end{equation*}
and the local Lagrangian dual function $q^t_{\text{IM}}(\mu)$ for stage $t \in \mathcal{T}$ is given by
\begin{equation*}
q^t_{\text{IM}}(\mu) := \min_{x^t} \left(
\sum_{i \in \mathcal{N}} \left( c^t_i + \mu'_i + \sum_{s=t}^T (\mu^+_{s} - \mu^-_{s}) \right) x^t_i
\ \middle| \ 
0 \leq x^t_i \leq u^t_i, \ i \in \mathcal{N}
\right).
\end{equation*}

For our evaluations, we choose the same parameter values as in \cite{Ben-Tal2004} that we repeat here for the sake of self-containedness. We consider a horizon of 48 weeks, where a production decision has to be made every two weeks (i.e., $T=24$). Moreover, there are $N = 3$ factories and the seasonal demand is given by
\begin{equation*}
d^t = 1000 \left( 1 + \frac{1}{2} \sin \left(\frac{\pi(t-1)}{12} \right) \right), \quad t \in \mathcal{T}.
\end{equation*}
The maximum production capacity per two-week period $u^t_i$ of each factory $i \in \mathcal{N}$ is 567 and the maximum production capacity over the entire horizon $C_i$ is 13,600. The minimum and maximum level of inventory at the warehouse is $L = 500$ and $U = 2000$ and the initial stock level $S$ is 500. Finally, we generate random production cost data for each factory $i \in \mathcal{N}$ by drawing each $c^t_i$ uniformly from the interval $[0.8\bar{c}^t_i , 1.2\bar{c}^t_i]$, where $\bar{c}^t_i$ is the expected value of the cost and given by
\begin{equation*}
\bar{c}^t_i = \bar{E}_i \left( 1 -\frac{1}{2} \sin \left(\frac{\pi(t-1)}{12} \right)\right), \quad t \in \mathcal{T},
\end{equation*}
where $\bar{E} := (1,1.5,2)$.

\subsection{Setup of the evaluation}
\label{sec_setup}

For both problem types introduced in Sections~\ref{sec_battery} and~\ref{sec_IM}, we study the performance of ODDO. In the case of Problem~BATTERY, for a given test instance, the corresponding ``training'' instances are created by taking the household consumption data of a given number of previous days as input for the cost functions. In the case of Problem~IM, the ``training'' instances are created by taking as input for the uncertain parameters in the objective function randomly generated values according to the distributions specified in Section~\ref{sec_IM}. For each training instance, we compute the optimal Lagrange multipliers. If for a given instance the optimal Lagrange multipliers are not unique (see Section~\ref{sec_Lagrange}), we take an arbitrary vector of optimal Lagrange multipliers. These values serve as the training data for predicting the optimal Lagrange multipliers of future instances. 

We carry out two sets of evaluations. First, we study the performance of four types of candidate multiplier vectors. For this, for a given test instance, let $\mathcal{H}$ denote the set of optimal Lagrange multipliers to the corresponding training instances. As four candidate vectors we choose the minimum, maximum, mean and median of the multipliers in $\mathcal{H}$. More precisely, for a given set $\mathcal{H}$, we compute the four candidate multiplier vectors $C^{\min} \equiv (\hat{\mu}^{\min},\hat{\lambda}^{\min})$, $C^{\max} \equiv (\hat{\mu}^{\max},\hat{\lambda}^{\max})$, $C^{\text{mean}} \equiv (\hat{\mu}^{\text{mean}},\hat{\lambda}^{\text{mean}})$, and $C^{\text{med}} \equiv(\hat{\mu}^{\text{med}},\hat{\lambda}^{\text{med}})$ by
\begin{align*}
C^{\min} \ : \ &&
\hat{\mu}^{\min}_j
&:= \min_{(\mu,\lambda) \in \mathcal{H}} \mu_j,
&\hat{\lambda}^{\min}_k
&:= \min_{(\mu,\lambda) \in \mathcal{H}} \lambda_k, \\
C^{\max} \ : \ &&
\hat{\mu}^{\max}_j
&:= \max_{(\mu,\lambda) \in \mathcal{H}} \mu_j, 
&\hat{\lambda}^{\max}_k
&:= \max_{(\mu,\lambda) \in \mathcal{H}} \lambda_k \\
C^{\text{mean}} \ : \ &&
\hat{\mu}^{\text{mean}}_j
&:= \frac{1}{|\mathcal{H}|} \sum_{(\mu,\lambda) \in \mathcal{H}} \mu_j,
&\hat{\lambda}^{\text{mean}} _k
&:= \frac{1}{|\mathcal{H}|} \sum_{(\mu,\lambda) \in \mathcal{H}} \lambda_k, \\
C^{\text{med}} \ : \ &&
\hat{\mu}^{\text{med}}_j
&:= \text{median}\lbrace \mu_j \ | \ (\mu,\lambda) \in \mathcal{H} \rbrace,
&\hat{\lambda}^{\text{med}}_k 
&:= \text{median}\lbrace \lambda_k \ | \ (\mu,\lambda) \in \mathcal{H} \rbrace.
\end{align*}
We consider two different sizes for the set of training instances, namely 10 and 50. Thus, for each test instance, we get two values for each of the four candidate multiplier vectors: one that corresponds to $|\mathcal{H}| = 10$ and one that corresponds to $|\mathcal{H}| = 50$. Overall, we evaluate the performance of these eight candidate values for ten test instances.

In the second set of evaluations, we assess the additional value of knowledge of the optimal Lagrange multipliers over knowledge of the uncertain data itself. To this end, we compare ODDO to a simple strategy that resembles optimization over nominal values. In this strategy, we obtain the online solution to a given test instance by replacing the uncertain parameters in the cost functions by the mean of the uncertain parameters of the training instances and solving the resulting deterministic optimization problem. More precisely, for a given test instance, let $\mathcal{P}$ denote the set of cost function parameters for the training instances. Subsequently, we solve the test instance where for Problem~BATTERY we take as value for the uncertain cost function parameters the values
\begin{equation*}
\hat{p}^t := \frac{1}{|\mathcal{P}|} \sum_{p \in \mathcal{P}} p
\end{equation*}
and for Problem~IM the values
\begin{equation*}
\hat{c}^t_i := \frac{1}{|\mathcal{P}|} \sum_{c \in \mathcal{P}} c^t_i, \quad i \in \lbrace 1,2,3 \rbrace,
\end{equation*}
where $c := (c^t_1, c^t_2, c^t_3)^{\top}_{t \in \mathcal{T}}$. We call this strategy the \emph{nominal} strategy since it represents optimization over the nominal values of the uncertain data. Alternatively, one could see this strategy as a simplified version of model predictive control (see, e.g., \cite{Grune2017}). We compare this nominal approach to ODDO where the multiplier prediction is taken as the optimal Lagrange multipliers of the problem solved in the nominal strategy, i.e., of the problem instance where the uncertain parameters are substituted by the values $\hat{p} := (\hat{p}^t)_{t \in \mathcal{T}}$ or $\hat{c} := (\hat{c}^t_1, \hat{c}^t_2, \hat{c}^t_3)^{\top}_{t \in \mathcal{T}}$. In this comparison, we consider four different sizes for the set of training instances, namely 1,~3,~5, and~10. We carry out the comparison for 50 problem instances.

\subsection{Results}
\label{sec_results}

In this section, we present and discuss the results of the evaluation as described in Section~\ref{sec_setup}. All simulations and computations are coded in Matlab using CVX \cite{Grant2008}. For each computed online solution, we calculated the ratio between the online and optimal offline objective value. This ratio serves as a measure for how good the online solution is: a ratio of 1 implies that the online solution is optimal. We discuss these results for the comparison of the four candidate multipliers $C^{\min}$, $C^{\max}$, $C^{\text{mean}}$, and $C^{\text{med}}$ in Section~\ref{sec_results_candidate} and for the comparison of ODDO and the nominal strategy in Section~\ref{sec_results_full}.

\subsubsection{Comparison of the four candidate multipliers}
\label{sec_results_candidate}

Figure~\ref{plot_battery_IM} shows the observed ratios for each of the four candidate multipliers. First, we compare the performance of the candidates~$C^{\min}$ and $C^{\max}$. It follows from Figures~\ref{plot_battery} and~\ref{plot_IM} that in all cases, candidate~$C^{\max}$ yields the worst performance out of all candidates and its ``opposite'' candidate~$C^{\min}$ performs significantly better than $C^{\max}$. This is in line with the robustness analysis in Section~\ref{sec_rob_errors}: when we under-predict the optimal Lagrange multipliers, the difference between the online and optimal offline objective value is bounded. In fact, it suggests that that the robustness result of Theorem~\ref{th_main} might apply to a broader class of instances than the ones specified in the class $I_{\text{sub}}$.

\begin{figure}[ht!]
\centering
\begin{subfigure}[b]{.48\textwidth}
\centering
{\includegraphics{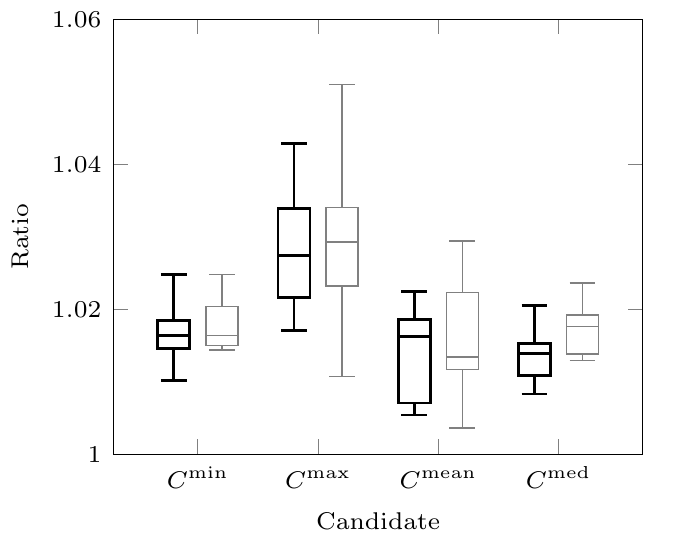}}
\caption{Problem~BATTERY.}
\label{plot_battery}
\end{subfigure}
\hspace{0.02\textwidth}
\begin{subfigure}[b]{.48\textwidth}
\centering
{\includegraphics{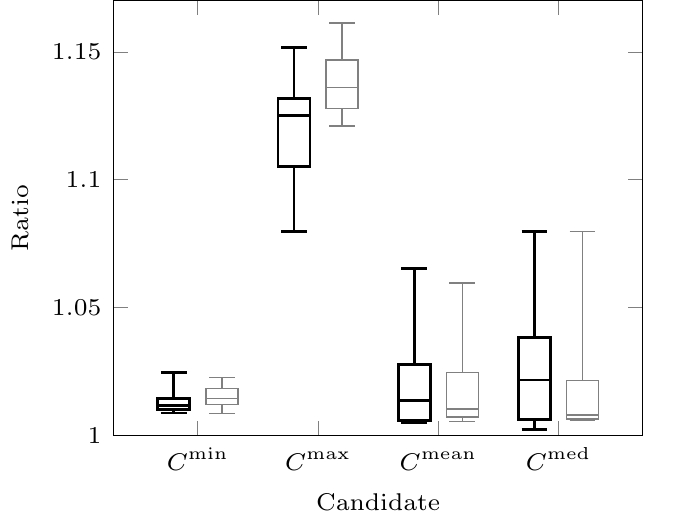}}
\caption{Problem~IM.}
\label{plot_IM}
\end{subfigure}
\caption{Boxplots of the observed ratios for the four candidate pblackictions, where for each candidate the black (left) boxplot corresponds to the case of 10 training instances and the gray (right) boxplot to the case of 50 training instances.}
\label{plot_battery_IM}
\end{figure}

For both problem types BATTERY and IM, at least one candidate has been able to compute an online solution whose objective value was at most 1\% worse than the optimal objective value ($C^{\min}$ for IM and $C^{\text{mean}}$ and $C^{\text{med}}$ for both problems). Moreover, candidates $C^{\text{mean}}$ and $C^{\text{med}}$ have been able to compute for most of their instances an online solution whose objective value was at most 1.77\% (Problem~BATTERY) and 1.24\% (Problem~IM) worse than the optimal objective value. This demonstrates that ODDO is capable of obtaining near-optimal solutions not only incidentally, but also on average in the long term.

In order to obtain information on the overall performance of the candidate multipliers, we focus on their median performance. For most of the considered cases in Problem~BATTERY, Figure~\ref{plot_battery} suggests that the median performance of the candidates $C^{\text{mean}}$ and $C^{\text{med}}$ is better than the median performance of candidates~$C^{\min}$ and $C^{\max}$. An explanation for this is that the candidates $C^{\text{mean}}$ and $C^{\text{med}}$ depend on all training data, whereas the candidate $C^{\min}$ depends solely on the value of (extreme) lower outliers. Thus, $C^{\text{mean}}$ and $C^{\text{med}}$ are more ``balanced'' and thus good candidates for the ``average'' instance.

Finally, we focus on the difference in performance between different sizes for the sets of training instances. Figure~\ref{plot_IM} suggests that, in Problem~IM, the median performance of candidates $C^{\min}$, $C^{\text{mean}}$, and $C^{\text{med}}$ does not differ significantly for the case of 10 training instances and for the case of 50 training instances. Interestingly however, Figure~\ref{plot_battery} indicates that for Problem~BATTERY these candidates seem to perform better in the case of 10 training instances than in the case of 50 training instances. One explanation for this could be that the latter case takes into account household data of approximately three months, whereas the former case is based on data from only the past one-and-a-half week. Since household consumption is heavily influenced by weather (see, e.g., \cite{Kavousian2013}), this means that also the weather from three months ago is taken into account when computing the candidate multiplier vectors. However, this historical weather has little to no influence on the current household consumption and thus also little to no influence on the optimal Lagrange multipliers corresponding to the current day. In earlier work \cite{SchootUiterkamp2018a}, we observed a similar relation between the optimal Lagrange multiplier and the choice of number of training instances.

\subsubsection{Comparison of ODDO to the nominal strategy}
\label{sec_results_full}

Figure~\ref{plot_compare} shows the observed ratios of the online solutions for Problem~BATTERY and Problem~IM for both ODDO and the nominal strategy. Moreover, Table~\ref{tab_MPC} shows the share of problem instances wherein ODDO outperforms the nominal strategy.

\begin{figure}[ht!]
\centering
\begin{subfigure}[t]{.41\textwidth}
\centering
{\includegraphics{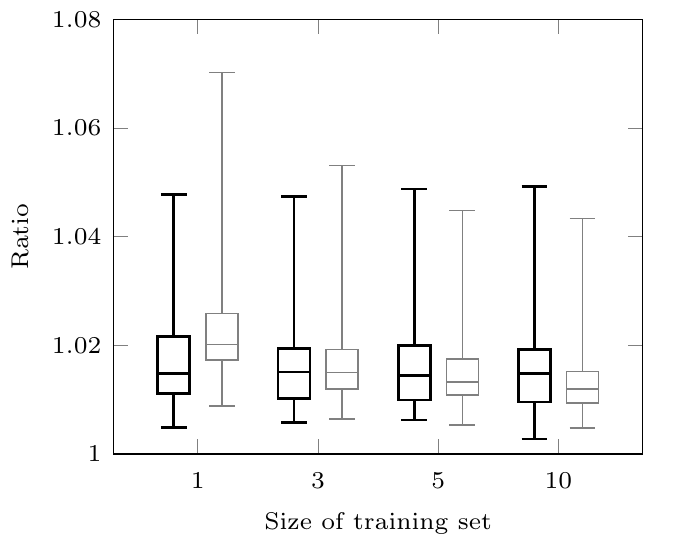}}
\caption{Problem~BATTERY; ODDO (left, black) and the nominal strategy (right, gray).}
\label{plot_battery_compare}
\end{subfigure}
\hspace{0.02\textwidth}
\begin{subfigure}[t]{.225\textwidth}
\centering
{\includegraphics{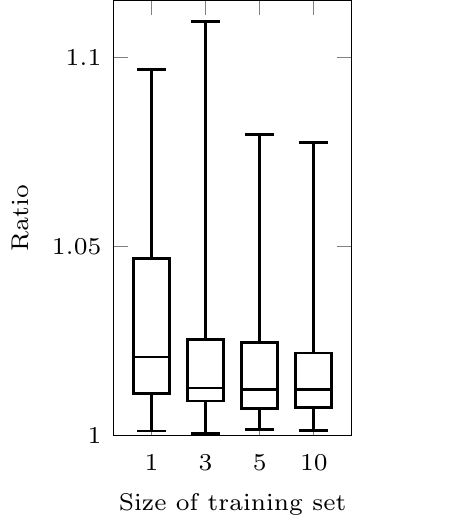}}
\caption{Problem~IM; ODDO.}
\label{plot_IM_compare_ODDO}
\end{subfigure}
\hspace{0.02\textwidth}
\begin{subfigure}[t]{.225\textwidth}
{\includegraphics{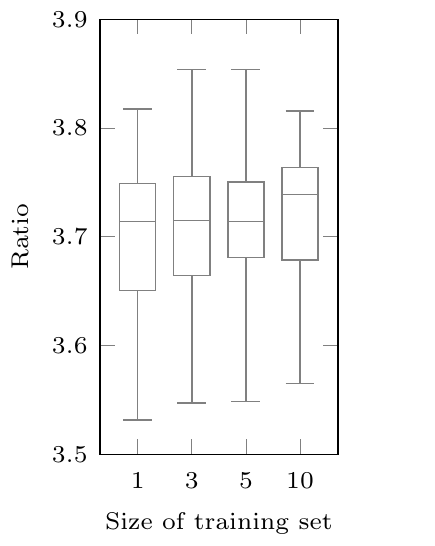}}
\caption{Problem~IM; nominal strategy.}
\label{plot_IM_compare_MPC}
\end{subfigure}

\caption{Boxplots of the observed ratio in the comparison between ODDO and the nominal strategy for each considered training set size.}
\label{plot_compare}
\end{figure}

\begin{table}[ht!]
\centering
\begin{tabular}{r | r r}
\toprule
&  \multicolumn{2}{c}{Success rate} \\
Size of training set & Problem~BATTERY & Problem~IM \\
\midrule
1 & 0.7 & 1\\
3 & 0.54 & 1\\
5 &  0.5 & 1\\
10 & 0.44 & 1\\
\bottomrule
\end{tabular}
\caption{Fraction of test instances wherein ODDO outperforms the nominal strategy.}
\label{tab_MPC}
\end{table}

For Problem~BATTERY, Table~\ref{tab_MPC} suggests that the share of problem instances wherein ODDO outperforms the nominal strategy decreases as the number of training instances increases. However, for Problem~IM, ODDO outperforms the nominal strategy in all cases by a large margin. The main cause of this seems to be the large difference in overall performance of the nominal strategy between Problem~BATTERY and Problem~IM: the ratios of this strategy for Problem~IM are significantly larger than for Problem~BATTERY, whereas the ratios of ODDO are similar for both problems. Additionally, for Problem~IM, in the nominal strategy, the overall behavior of the ratios does not significantly change as the number of training instances increases, whereas for ODDO increasing the number of training instances leads to a rapid decrease and a reduced spread of the ratios. One explanation for this difference is that the uncertain data in Problem~IM are mutually independent both within one instance and between multiple instances. As a consequence, the training instances do not provide additional information on the data structure other than an improved estimate of the variance of the data. Another explanation is that in Problem~IM the relative variance of the uncertain data is larger than in Problem~BATTERY. More precisely, for Problem~BATTERY the (estimated) coefficient of variation (as calculated according to \cite{Albert2010}) of the uncertain data of the 50 test instances is only 0.1018, whereas these coefficients for the uncertain data $c_1 := (c^t_1)_{t \in \mathcal{T}}$, $c_2 := (c^t_2)_{t \in \mathcal{T}}$, and $c_3 := (c^t_3)_{t \in \mathcal{T}}$ of Problem~IM are~0.1450,~0.2175, and~0.2900 respectively. This supports the claim made in Section~\ref{sec_mot} that the performance of ODDO is influenced only limitedly by large variances and unanticipated realizations of the uncertain data.

Lastly, we focus on the influence of the number of training instances on the performance of both ODDO and the nominal strategy. On the one hand, for ODDO, the results in Figures~\ref{plot_battery_compare} and~\ref{plot_IM_compare_ODDO} indicate that the ratios are only marginally affected by the chosen number of training instances. On the other hand, for the nominal strategy, Figure~\ref{plot_battery_compare} suggests that the ratios for Problem~BATTERY decrease with an increasing number of training instances and become competitive with ODDO only when the number of training instances is five or more. This suggests that ODDO can obtain good online solutions with a small training set and that these solutions are better than those produced by the nominal strategy.

Summarizing, the evaluation in this section shows that ODDO is able to achieve near-optimal results using relatively easy-to-compute candidates as input for both problems with real (Problem~BATTERY) and randomly generated (Problem~IM) data. Moreover, the results for Problem~IM indicate that ODDO performs well in practice also for problems that do not fall into the class $I_{\text{sub}}$. This suggests the existence of a wider class of problems for which ODDO yields good online solutions. We plan to further investigate such a broader class in future research.

\section{Discussion}
\label{sec_disc}

In this section, we discuss some limitations and possible extensions of the ODDO-framework. We provide several suggestions for overcoming part of these limitations and point out several interesting and necessary directions of future research.

\subsection{High-dimensional dual spaces}
In part, ODDO exploits the fact that, for many problems from practice, the dimension of the dual space is smaller than the dimension of the primal space. However, there are optimization problems where this is not the case, e.g., where the number of constraints is exponential in the number of variables. In some cases, though, we can reduce the number of to-be-predicted values by aggregating some of the dual variables. For example, suppose that each of the functions $g^t_j$ and $h^t_k$ is a constant multiple of some function $\tilde{g}^t$ and $\tilde{h}^t$ respectively, i.e., $g^t_{j} (x^t) = \tilde{a}^t_j \tilde{g}^t(x^t)$ and $h^t_k (x^t) = \tilde{b}^t_k h^t (x^t)$ for some $\tilde{g}^t_j, \tilde{h}^t_k \geq 0$. Note that this includes the case where the functions $g^t_j$ are linear. Then the local Lagrangian is given by
\begin{equation*}
L^t(x^t,\mu,\lambda) = f^t(x^t) + \sum_{j \in \mathcal{M}} \mu_j \tilde{a}^t_j \tilde{g}^t (x^t) + \sum_{k \in \mathcal{L}} \lambda_k \tilde{b}^t_k h^t(x^t).
\end{equation*}
Observe that predictions of the aggregated terms $ \sum_{j \in \mathcal{M}} \mu_j \tilde{a}^t_j$ and $\sum_{k \in \mathcal{L}} \lambda_k \tilde{b}^t_k$ instead of the individual optimal Lagrange multipliers $(\mu,\lambda)$ are sufficient to solve the online problems $(P^{\bar{t}}(\mu,\lambda))$ and $(P^{\bar{t}}_{\text{proj}}(\mu,\lambda))$. Thus, this allows to reduce the number of to-be-predicted values from $|\mathcal{M}| + |\mathcal{L}| = O\left(2^T\right)$ to $2T$.

\subsection{Uncertainty in the constraints}
In our research, we focus on optimization problems where only the objective function is uncertain. In contrast, the existing paradigms of robust optimization and stochastic programming can accommodate uncertainty in both the objective and the constraints. Note that the structure of Problem~\prob\ allows for problems where during each stage $t \in \mathcal{T}$ the constraint functions $g^t_j$ and $h^t_k$ are revealed. However, in the current framework, the corresponding online decision for this stage cannot be determined such that future feasibility is preserved because the constraint functions for the future stages are unknown. To overcome this problem, one might be able to approximate the uncertain constraints by using techniques from, e.g., stochastic programming, (adjustable) robust optimization, or online convex optimization, depending on which information on the type of constraint uncertainty is available. Combining such an approximation with ODDO is an interesting direction for future research.

\subsection{Problem structure}

Our approach requires the feasible set of Problem~\prob\ to be convex and the cost functions $f^t$ to be continuously differentiable and \emph{strictly} convex. This restriction ensures that properties such as strong duality hold and enables us to derive the robustness results of Section~\ref{sec_rob}. In this section, we briefly consider how the validity of these properties and results are affected when we relax some of the restrictions on the feasible set and the cost functions.

First, suppose that one or more of the cost functions are not \emph{strictly} convex and / or not continuously differentiable. We already considered such a problem, namely Problem~IM. In this case, it can happen that not only the optimal Lagrange multipliers might be non-unique, but also the optimal primal solution. Thus, when solving a subproblem during the online optimization process, we might have to choose from a set of multiple online solutions. In our evaluation, we did not include any preference for or requirement of an online solution from this set, but simply chose the solution computed by the solver. It would be interesting to investigate under which conditions it is useful to incorporate such a preference or requirement and what the effect of this is on the overall online optimization process.

Second, suppose that one or more of the cost functions are not convex. As a consequence, strong duality is not guaranteed anymore, which implies that there might not exist Lagrange multipliers $(\mu^*,\lambda^*)$ such that $x(\mu^*,\lambda^*)$ is optimal for the original Problem~\prob. However, it is known that strong duality holds for several optimization problems with a nonconvex objective function, in particular for several problems related to energy management (see, e.g., \cite{Flores-Bazan2014}, \cite{Lavaei2012}). One direction for future research is to apply ODDO to these types of problems and, e.g., asses the validity of the robustness results of Section~\ref{sec_rob_errors} for these problems.

Third, suppose that one or more of the decision variables is binary or integer. This implies that the feasible region of Problem~\prob\ is not convex. In the area of energy management, binary and integer decision variables often arise when a device has several operation modes, e.g., an EV that has a finite number of possible charging rates (see, e.g., \cite{Yilmaz2013}). Recently, we derived an optimal solution characterization and an online optimization approach similar to ODDO specifically for EV scheduling with binary state-switching variables by exploiting the specific structure of the problem  \cite{SchootUiterkamp2018b}. It would be interesting to gain more insight into why this structure could be exploited and how this result can be extended to other (mixed-)integer optimization problems.

\subsection{Prediction of the optimal Lagrange multipliers}
An important direction for future research is to find a structural and unifying approach to predict the optimal Lagrange multipliers in our approach. In Section~\ref{sec_results_candidate}, we have shown that simple statistics such as the mean and median of previously observed optimal Lagrange multipliers can in practice perform well as multiplier predictions. However, for more general problems and, in particular, problems with multiple equality constraints, their predictive quality decreases significantly. Thus, research is needed on the behavior of the optimal Lagrange multipliers and on a more sophisticated prediction approach that is suitable for more general problem instances.

One possible starting point for this is to use the analysis in Section~\ref{sec_rob} and in particular Theorem~\ref{th_main} to obtain a measure of what constitutes a good (enough) prediction. This analysis suggests a preference for under-predicting the optimal Lagrange multipliers rather than over-predicting, since in the former case one can bound the difference in objective between the online and optimal offline solution. The evaluation results in Section~\ref{sec_results} confirm this preference, also for problem instances that do not satisfy the requirements of, e.g., Theorem~\ref{th_main}. Predicting values such that they are preferably beneath a (given) threshold value corresponds to the concept of quantile functions in statistics \cite{Serfling2002}. In earlier work on scheduling the charging of EVs \cite{SchootUiterkamp2018a}, we used this concept to successfully predict optimal Lagrange multipliers (also called ``fill-levels'' in this application). Thus, it is worthwhile to investigate the possibility of generalizing this approach to the general Problem~\prob.

Another possible starting point, in particular for the subclass $I_{\text{sub}}$, is to exploit the problem structure and thereby find a useful relation between the optimal Lagrange multipliers and (a function of) the uncertain data. As an example of this, we consider the resource allocation problem RAP from Section~\ref{sec_Lagrange} where the cost functions are of the form $a^i \bar{f} (\frac{x^i}{a^i} + b^i)$ for a known convex differentiable function $\bar{f}$, a known vector $a \in \mathbb{R}^n_{>0}$, and an uncertain vector $b \in \mathbb{R}^n$ (see also Section~\ref{sec_rob}). If in the optimal solution we have that each variable is strictly in between its bounds, i.e., we have $l^i < x^i(\mu < u^i$ for all $i \in \lbrace 1,\ldots, n \rbrace$, the optimal solution is equal to the local Lagrangian solution evaluated for the optimal Lagrange multiplier~$\lambda^*$, i.e., $x^i(\lambda^*) = a^i ((\nabla \bar{f})^{-1}(\lambda^*) - b^i)$ (see also the proof of Corollary~\ref{col_main_sep}). Summing the terms $x^i(\lambda^*)$ over $i$ yields
\begin{equation*}
R = \sum_{i = 1}^n x^i (\lambda^*)
= \sum_{i=1}^n a^i ((\nabla \bar{f})^{-1}(\lambda^*) - b^i)
= (\nabla \bar{f})^{-1}(\lambda^*) \sum_{i=1}^n a^i - \sum_{i=1}^n a^i b^i.
\end{equation*}
It follows that
\begin{equation}
\lambda^* = \nabla \bar{f} \left( \frac{R + \sum_{i=1}^n a^i b^i}{\sum_{i=1}^n a^i} \right) .
\label{eq_cand_RAP}
\end{equation}
Note that $\lambda^*$ does not depend explicitly on each individual element of the uncertain vector $b$ but only on the aggregated term $\sum_{i=1}^n a^i b^i$. This suggests that we can obtain a promising prediction for $\lambda^*$ using a prediction of this aggregated term via Equation~(\ref{eq_cand_RAP}) instead of a prediction of each individual element of~$b$. By the law of large numbers, accurately predicting such an aggregated term is in general easier than predicting each individual element. For instance, in the Problem BATTERY, the aggregated term $\sum_{t \in \mathcal{T}} a^t b^t$ corresponds to the average power consumption over the intervals $\mathcal{T}$. This average power consumption is in general significantly easier to predict than the power consumption of each individual interval (see, e.g., \cite{Javed2012}).

The relation between the optimal Lagrange multiplier and the uncertain vector~$b$ in Equation~(\ref{eq_cand_RAP}) is particularly interesting for the problem class $I_{\text{sub}}$. This is because for each problem of this class, there exists a partition of the decision variables such that the problem can be equivalently decomposed into a collection of RAPs, one over each element of the partition (see, e.g., \cite{Fujishige2005}). This partition corresponds directly with the set of constraints that are tight in the optimal primal solution and thereby with the set of nonzero multipliers in the optimal Lagrange multiplier vector. If we are able to accurately predict this partition, we can also obtain a promising prediction of the optimal Lagrange multiplier for each RAP subproblem using Equation~(\ref{eq_cand_RAP}).

\section{Conclusions}
\label{sec_concl}

We presented a new framework for optimization under uncertainty called ``Online Duality-Driven Optimization'' (ODDO). This framework is motivated by applications in energy management for micro-grids, where peak energy consumption needs to be minimized in order to maintain a proper operation of the micro-grid. The presented approach does not require any quantitative assumptions on the uncertain data involved in the problem such as uncertainty sets or probability distributions. The key idea of the framework is to predict the optimal Lagrange multipliers of the optimization problem instead of the actual uncertain data or objective function. We analyzed the robustness of this approach both in theory and in practice. For a specific but important class of problems, we derived bounds on the difference in objective value between the online and optimal offline solution. Moreover, evaluations suggest that in practice this robustness carries over to problems that fall outside of this class. For the studied problems, simple statistics such as the mean and median of previously observed optimal Lagrange multipliers seem to perform well as input predictions for the ODDO-framework. 

In future work, we aim to improve and extend the scope of the ODDO-framework according to the aspects in Sections~\ref{sec_results} and~\ref{sec_disc}. Apart from the current work on the framework itself, we plan to apply the framework to other problems in energy management for micro-grids. Summarizing, we believe that the ODDO-framework is a promising addition to the set of paradigms for optimization under uncertainty and provides the community with a new approach to tackle these types of problems.

\section*{Acknowledgments}
The authors would like to thank Dick den Hertog and Pierre Pinson for their comments on and suggestions for the positioning of the main idea in this article. This research has been conducted within the SIMPS project (647.002.003) supported by NWO and Eneco.

\appendix

\section{Proofs}

\subsection{Proof of Lemma~\ref{lemma_tight}}
\label{app_lemma_tight}

\begin{lemma_tight}
For any instance in $I_{\text{sub}}$, a given $\bar{t} \in \mathcal{T}$, and given multipliers $(\mu,\lambda)$ and corresponding online solution $\hat{x}^{\bar{t}}(\mu,\lambda)$ and local Lagrangian solution $x^{\bar{t}}(\mu,\lambda)$, we have:
\begin{equation*}
\hat{x}^{\bar{t}}(\mu,\lambda) < x^{\bar{t}}(\mu,\lambda)
\Longrightarrow 
\exists \bar{\mathcal{S}}^{\bar{t}} \subseteq \left\{ 1,\ldots,\bar{t} \right\} \text{ with } \bar{\mathcal{S}}^{\bar{t}} \ni \bar{t} \text{ such that } \sum_{t \in \bar{\mathcal{S}}^{\bar{t}}} \hat{x}^t(\mu,\lambda) = r\left(\bar{\mathcal{S}}^{\bar{t}} \right).
\end{equation*}
\end{lemma_tight}

\begin{proof}

Suppose that $\hat{x}^{\bar{t}} (\mu,\lambda) < x^{\bar{t}} (\mu,\lambda)$. Since the function $f^{\bar{t}}$ is \emph{strictly} convex, the local Lagrangian solution $x^{\bar{t}}(\mu,\lambda)$ is the \emph{unique} optimal solution to the inner problem of the local Lagrangian dual function~$q^{\bar{t}}(\mu,\lambda)$ and thus $L^{\bar{t}} (x^{\bar{t}}(\mu,\lambda),\mu,\lambda) < L^{\bar{t}} (\hat{x}^{\bar{t}}(\mu,\lambda),\mu,\lambda)$. This implies that $L^{\bar{t}} (\cdot,\mu,\lambda)$ is decreasing on the interval $[\hat{x}^{\bar{t}}(\mu,\lambda) , x^{\bar{t}} (\mu,\lambda)]$ since any solution in this interval is feasible for the inner problem of $q^{\bar{t}}(\mu,\lambda)$ and $x^{\bar{t}}(\mu,\lambda)$ is optimal for this problem. 

Consider a vector $y \in \mathbb{R}^T$ with $y^t = \hat{x}^t(\mu,\lambda)$ for $t < \bar{t}$ and $y^{\bar{t}} \in ( \hat{x}^{\bar{t}}(\mu,\lambda), x^{\bar{t}}(\mu,\lambda))$. If this vector would have been a feasible solution, the value $y^{\bar{t}}$ had been chosen as online solution for stage~$\bar{t}$ since $L^{\bar{t}}(y^{\bar{t}},\mu,\lambda) < L^{\bar{t}}(\hat{x}^{\bar{t}},\mu,\lambda)$. In particular, this means that there is no $\epsilon \in (0 , x^{\bar{t}}(\mu,\lambda) - \hat{x}^{\bar{t}}(\mu,\lambda))$ such that for some $s > \bar{t}$ the solution wherein we move a positive amount of $\epsilon$ from $\hat{x}^s(\mu,\lambda)$ to $\hat{x}^{\bar{t}}(\mu,\lambda)$, i.e., the solution $\bar{v} := (\bar{v}^t)_{t \in \mathcal{T}}$ given by
\begin{equation*}
\bar{v}^t = \left\{
\begin{array}{ll}
\hat{x}^{\bar{t}}(\mu,\lambda) + \epsilon & \text{if } t = \bar{t}, \\
\hat{x}^{s}(\mu,\lambda) - \epsilon & \text{if } t = s, \\
\hat{x}^t(\mu,\lambda) & \text{otherwise,}
\end{array}
\right.
\end{equation*}
is feasible. Since $\sum_{t \in \mathcal{T}} \bar{v}^t) = r(\mathcal{T})$, this implies that there is at least one inequality constraint that is violated by $\bar{v}$. Since $\epsilon$ may be arbitrary close to $0$, it follows for the online solution $\hat{x}(\mu,\lambda)$ that there is at least one inequality constraint involving $\hat{x}^{\bar{t}}(\mu,\lambda)$ but not involving $\hat{x}^s (\mu,\lambda)$ that is tight in this online solution. In other words, for each $s > \bar{t}$, there is a subset $\hat{\mathcal{V}}^s \subset \mathcal{T}$ such that $\bar{t} \in \hat{\mathcal{V}}^s$, $s \not\in \hat{\mathcal{V}}^s$, and $\sum_{t \in \hat{\mathcal{V}}^s} \hat{x}^t (\mu,\lambda) = r(\hat{\mathcal{V}}^s)$. We denote the intersection of these sets by $\hat{\mathcal{V}}$, i.e., $\hat{\mathcal{V}} := \cap_{s > \bar{t}} \hat{\mathcal{V}}^s$. Note that $\hat{\mathcal{V}}$ is not empty since each set $\hat{\mathcal{V}}^s$ contains $\bar{t}$.

We claim that the set $\hat{\mathcal{V}}$ satisfies all properties of $\bar{\mathcal{S}}^{\bar{t}}$ that are required by the lemma, which implies that such a set exists. First, we have that $\bar{t} \in \hat{\mathcal{V}}$ since each subset $\hat{\mathcal{V}}^s$ contains $\bar{t}$. Second, since for each $s > \bar{t}$ the subset $\hat{\mathcal{V}}^s$ does not contain $s$, the intersection $\hat{\mathcal{V}}$ cannot contain any indices exceeding $\bar{t}$. Third, since for each $s > \bar{t}$ the constraint corresponding to $\hat{\mathcal{V}}^s$ is tight in the online solution $\hat{x}(\mu,\lambda)$, Lemma~\ref{lemma_sub} implies that also the constraint corresponding to $\hat{\mathcal{V}}$ is tight in $\hat{x}(\mu,\lambda)$.
\end{proof}

\subsection{Proof of Lemma~\ref{lemma_prop_2}}
\label{app_lemma_prop_2}

\begin{lemma_prop_2}
For any instance in $I_{\text{sub}}$ and given any Lagrange multipliers $(\mu,\lambda)$ such that $(\mu,\lambda) \leq (\mu^*,\lambda^*)$, it holds that $\hat{x}^t (\mu,\lambda) \leq x^t(\mu^*,\lambda^*)$ for all $t \in \mathcal{T}$.
\end{lemma_prop_2}

\begin{proof}
Consider any Lagrange multipliers $(\mu,\lambda)$ such that $(\mu,\lambda) \leq (\mu^*,\lambda^*)$. We prove the lemma by induction on the stage index~$\bar{t}$. First, we consider the case $\bar{t} = 1$. Observe that, when defining the rank function $r$, we can assume without loss of generality that $r(\lbrace 1 \rbrace)$ is the maximum value of $x^1$ in any feasible solution $x$ in $\mathcal{F}$. As a consequence, the local Lagrangian solution $x^1(\mu,\lambda)$ cannot lead to any future constraint violation. Hence, $\hat{x}^1(\mu,\lambda) = x^1(\mu,\lambda)$, which proves the case $\bar{t} = 1$.

Second, suppose that $\hat{x}^s(\mu,\lambda) \leq x^s(\mu,\lambda)$ for all $s$ smaller than some $\bar{t} \in \mathcal{T}$ with $\bar{t} > 1$. We prove that this implies $\hat{x}^{\bar{t}}(\mu,\lambda) \leq x^{\bar{t}} (\mu,\lambda)$ by contradiction. For this, suppose that $\hat{x}^{\bar{t}}(\mu,\lambda) > x^{\bar{t}}(\mu,\lambda)$. It follows from an argument analogous to the first part of the proof of Lemma~\ref{lemma_tight} in Appendix~\ref{app_lemma_tight} that there is no $\epsilon$ with $0 < \epsilon < \hat{x}^{\bar{t}}(\mu,\lambda) - x^{\bar{t}}(\mu,\lambda)$ such that for some $s > \bar{t}$ the solution wherein we move an amount of $\epsilon$ from $\hat{x}^{\bar{t}}(\mu,\lambda)$ to $\hat{x}^{s}(\mu,\lambda)$, i.e., the solution $\bar{w} := (\bar{w}^t)_{t \in \mathcal{T}}$ given by
\begin{equation*}
\bar{w}^t = \left\{
\begin{array}{ll}
\hat{x}^{\bar{t}}(\mu,\lambda) - \epsilon & \text{if } t = \bar{t}, \\
\hat{x}^{s}(\mu,\lambda) + \epsilon & \text{if } t = s, \\
\hat{x}^t(\mu,\lambda) & \text{otherwise,}
\end{array}
\right.
\end{equation*}
is infeasible. Since $\sum_{t \in \mathcal{T}} \bar{w}^t = r(\mathcal{T})$, this implies for the online solution $\hat{x}(\mu,\lambda)$ that there must be at least one inequality constraint involving $\hat{x}^s(\mu,\lambda)$ but not involving $\hat{x}^{\bar{t}}(\mu,\lambda)$ that is tight in this online solution. In other words, for each $s > \bar{t}$, there exists a subset $\hat{\mathcal{S}}^s \subset \mathcal{T}$ such that $s \in \hat{\mathcal{S}}^s$, $\bar{t} \not\in \hat{\mathcal{S}}^s$, and $\sum_{t \in \hat{\mathcal{S}}^s} \hat{x}^t(\mu,\lambda) = r(\hat{\mathcal{S}}^s)$. 

Let $\bar{\mathcal{T}}$ denote the set of stages whose local Lagrangian and online solution are not the same, i.e., $\bar{\mathcal{T}} := \lbrace t \in \mathcal{T} \ | \ \hat{x}^t(\mu,\lambda) \neq x^t(\mu,\lambda) \rbrace$. By the induction hypothesis and Lemma~\ref{lemma_tight}, it follows that for each $s \in \bar{\mathcal{T}}$ with $s < \bar{t}$ there exists a set $\bar{\mathcal{S}}^{s} \subseteq \lbrace 1,\ldots,s \rbrace$ with $s \in \bar{\mathcal{S}}^{s}$ such that $\sum_{s \in \bar{\mathcal{S}}} \hat{x}^s (\mu,\lambda) = r(\bar{\mathcal{S}}^{s})$. Thus, by Lemma~\ref{lemma_sub}, also the inequality constraint corresponding to the set
\begin{equation*}
\mathcal{S}' := \left( \bigcup_{s \in \bar{\mathcal{T}}, s< \bar{t}} \bar{\mathcal{S}}^{s} \right) \cup  \left( \bigcup_{s > \bar{t}} \hat{\mathcal{S}}^s \right)
\end{equation*}
is tight in the online solution $\hat{x}(\mu,\lambda)$, i.e., $\sum_{t \in \mathcal{S}'} \hat{x}^{t}(\mu,\lambda) = r(\mathcal{S}')$. Note, that 
\begin{itemize}
\item 
$\lbrace \bar{t} + 1,\ldots, T \rbrace \subseteq \mathcal{S}'$ since for each $s > \bar{t}$ the set $\hat{\mathcal{S}}^s$ contains~$s$;
\item
$\bar{t} \not\in \mathcal{S}'$ since for each~$s \in \bar{\mathcal{T}}$ with $s < \bar{t}$ we have $\bar{\mathcal{S}}^s \subseteq \lbrace 1,\ldots,s \rbrace$ and for each $s > \bar{t}$ we have by definition that $\bar{t} \not\in \hat{\mathcal{S}}^s$;
\item
$s \in \mathcal{S}'$ for all $s \in \bar{\mathcal{T}}$ with $s < \bar{t}$ since each set $\bar{\mathcal{S}}^{s}$ contains~$s$.
\end{itemize}
It follows that $(\mathcal{T} \backslash \mathcal{S}') \subseteq \lbrace 1,\ldots,\bar{t} \rbrace$, $\bar{t} \in (\mathcal{T} \backslash \mathcal{S}')$, and $(\mathcal{T} \backslash \mathcal{S}') \cap \bar{\mathcal{T}} = \emptyset$. We use this information to derive the following inequality:
\begin{subequations}
\begin{align}
r(\mathcal{T}) - r(\mathcal{S}')
&= \sum_{t \in \mathcal{T} \backslash \mathcal{S}'} \hat{x}^t(\mu,\lambda) \label{eq_sub_01} \\
&=
\sum_{t \in \mathcal{T} \backslash \mathcal{S}', t \neq \bar{t}} x^t(\mu,\lambda) + \hat{x}^{\bar{t}}(\mu,\lambda)  \label{eq_sub_02} \\
&\geq
\sum_{t \in \mathcal{T} \backslash \mathcal{S}', t \neq \bar{t}} x^t(\mu^*,\lambda^*) + \hat{x}^{\bar{t}}(\mu,\lambda)  \label{eq_sub_03} \\
&=
r(\mathcal{T}) - \sum_{t \in \mathcal{S}'} x^t(\mu^*,\lambda^*) - x^{\bar{t}}(\mu^*,\lambda^*) + \hat{x}^{\bar{t}}(\mu,\lambda). \label{eq_sub_04} \\
&\geq r(\mathcal{T}) - r(\mathcal{S}') - x^{\bar{t}}(\mu^*,\lambda^*) + \hat{x}^{\bar{t}}(\mu,\lambda).  \label{eq_sub_05}
\end{align} \label{eq_sub}%
\end{subequations}
Here,~(\ref{eq_sub_01}) follows since the constraints corresponding to $\mathcal{T}$ and $\mathcal{S}'$ are tight in $\hat{x}(\mu,\lambda)$, i.e., $\sum_{t \in \mathcal{T}} \hat{x}^t(\mu,\lambda) = r(\mathcal{T})$ and $\sum_{t \in \mathcal{S}'} \hat{x}^t(\mu,\lambda) = r(\mathcal{S}')$,~(\ref{eq_sub_02}) follows since $t \not\in \bar{\mathcal{T}}$ for each $t \in \mathcal{T} \backslash \mathcal{S}'$ with $t < \bar{t}$,~(\ref{eq_sub_03}) follows by Lemma~\ref{lemma_instance}, and~(\ref{eq_sub_04}) and~(\ref{eq_sub_05}) follow due to feasibility of the optimal solution $x(\mu^*,\lambda^*)$ for the submodular constraints. The inequality derived in~(\ref{eq_sub}) implies that $x^{\bar{t}}(\mu^*,\lambda^*) \geq \hat{x}^{\bar{t}}(\mu,\lambda)$. However, since we assumed that $\hat{x}^{\bar{t}}(\mu,\lambda) > x^t(\mu,\lambda)$, this implies that $x^{\bar{t}}(\mu^*,\lambda^*) > x^{\bar{t}}(\mu,\lambda)$, which is a contradiction with Lemma~\ref{lemma_instance}. Thus, we must have that $\hat{x}^{\bar{t}}(\mu,\lambda) \leq x^{\bar{t}}(\mu,\lambda)$.
\end{proof}

\section{Computing the projection sets $\mathcal{C}^t_{\text{proj}}$ for Problem~BATTERY}
\label{app_battery}

Suppose that we need to compute the online decision $x^{\bar{t}}(\mu,\lambda)$ at the start of the interval $\bar{t}$. To compute the projection set $\mathcal{C}^{\bar{t}}_{\text{proj}}$, we iteratively eliminate the variables $x^T,x^{T-1},\ldots,x^{\bar{t} + 1}$ from the constraints~(\ref{battery_02})-(\ref{battery_04}). First, we eliminate $x^T$. This yields the following inequalities on the remaining non-fixed variables $x^{\bar{t}}, \ldots,x^{T-1}$:
\begin{align*}
& \underline{C}^t \leq \Delta t \sum_{s=1}^{\bar{t}-1} \hat{x}^s(\mu,\lambda) + \Delta t  \sum_{s = \bar{t}}^{t} x^s \leq \bar{C}^t, \quad t \in \lbrace \bar{t},\ldots,T-1 \rbrace, \\
& C - u^T \leq \Delta t  \sum_{s=1}^{\bar{t} - 1} \hat{x}^s(\mu,\lambda) + \Delta t  \sum_{s=\bar{t}}^{T-1} x^s \leq C - l^T, \\
& l^t \leq x^t \leq u^t, \quad t \in \lbrace \bar{t},\ldots, T-1 \rbrace.
\end{align*}
Observe that there are now two lower and upper bound constraints on the sum $ \Delta t \sum_{s=1}^{T-1} x^s$. We merge them by defining $\hat{\underline{C}}^{T-1} := \max(\underline{C}^{T-1} , C - u^T)$ and $\hat{\bar{C}}^{T-1} := \min (\bar{C}^{T-1} , C - l^T)$. Thus, we obtain
\begin{align*}
& \underline{C}^t \leq \Delta t  \sum_{s=1}^{\bar{t} - 1} \hat{x}^s(\mu,\lambda) + \Delta t  \sum_{s=\bar{t}}^{t} x^s \leq \bar{C}^t, \quad t \in \lbrace \bar{t} ,\ldots,T-2 \rbrace , \\
& \hat{\underline{C}}^{T-1} \leq \Delta t  \sum_{s=1}^{\bar{t}-1} \hat{x}^s(\mu,\lambda) + \Delta t  \sum_{s=\bar{t}}^{T-1} x^s \leq \hat{\bar{C}}^{T-1}, \\
& l^t \leq x^t \leq u^t, \quad t \in \lbrace \bar{t}, \ldots, T-1 \rbrace.
\end{align*}
Now we eliminate the next variable $x^{T-1}$. For this, we can apply the same reasoning as in the elimination of $x^T$. This yields the following inequalities:
\begin{align*}
& \underline{C}^t \leq \Delta t  \sum_{s=1}^{\bar{t} - 1} \hat{x}^s(\mu,\lambda) + \Delta t  \sum_{s=\bar{t}}^{t} x^s \leq \bar{C}^t, \quad t \in \lbrace \bar{t} ,\ldots,T-3 \rbrace, \\
& \hat{\underline{C}}^{T-2} \leq \Delta t  \sum_{s=1}^{\bar{t}-1} \hat{x}^s(\mu,\lambda) + \Delta t  \sum_{s=\bar{t}}^{T-2} x^s \leq \hat{\bar{C}}^{T-2}, \\
& l^t \leq x^t \leq u^t, \quad t \in \lbrace \bar{t}, \ldots, T-2 \rbrace,
\end{align*}
where $\hat{\underline{C}}^{T-2} := \max(\underline{C}^{T-2}, \hat{\underline{C}}^{T-1} - u^{T-1})$ and $\hat{\bar{C}}^{T-2} := \min(\bar{C}^{T-2}, \hat{\bar{C}}^{T-1} - l^{T-1})$. We can continue this process until all variables $x^t$ with $t > \bar{t}$ have been eliminated. The only remaining constraint is then
\begin{equation*}
\max \left(\hat{\underline{C}}^{\bar{t}} -  \Delta t \sum_{s=1}^{\bar{t} -1} \hat{x}^s(\mu,\lambda) , l^{\bar{t}} \right)
\leq x^{\bar{t}}
\leq
\min \left(\hat{\bar{C}}^{\bar{t}} -  \Delta t \sum_{s=1}^{\bar{t} - 1} \hat{x}^s(\mu,\lambda), u^{\bar{t}} \right),
\end{equation*}
where $\hat{\underline{C}}^{t} := \max(\underline{C}^{t}, \hat{\underline{C}}^{t+1} - u^{t+1})$ and $\hat{\bar{C}}^{t} := \min(\bar{C}^{t}, \hat{\bar{C}}^{t+1} - l^{t+1})$ for $t \in \lbrace 1,\ldots,T-1 \rbrace$. Thus, the projected problem is the minimization of a quadratic function over a closed interval whose boundaries can be computed efficiently. Moreover, observe that we can compute the entire vectors $\hat{\underline{C}}$ and $\hat{\bar{C}}$ in $O(T)$ time and at the start of the first interval since they do not depend on any of the online decisions $x^t (\mu,\lambda)$.

\bibliographystyle{abbrv}
\bibliography{ODDO_library}

%\end{multicols}
\end{document}